\documentclass[10pt]{amsart}

\usepackage{geometry}
\usepackage{amssymb,latexsym, amscd, amsthm, amsmath}
\input xypic
\xyoption{all}

\newtheorem{theorem}{Theorem}[section]
\newtheorem{definition}[theorem]{Definition}
\newtheorem{proposition}[theorem]{Proposition}
\newtheorem{lemma}[theorem]{Lemma}

\newtheorem{corollary}[theorem]{Corollary}
\newtheorem{conjecture}[theorem]{Conjecture}

\newtheorem{remark}[theorem]{Remark}

\begin{document}

\mathchardef\smallp="3170

\title[$\ell$--adic realizations of Picard $1$--motives] {{The Galois module structure of $\ell$--adic realizations\\ of
Picard $1$--motives and applications}}
\author[C. Greither and C. Popescu]{Cornelius Greither and Cristian D. Popescu}

\address{Institut f\"ur Theoretische Informatik und Mathematik,
Universit\"at der Bundeswehr, M\"unchen,
85577 Neubiberg, Germany}
\email{cornelius.greither@unibw.de}

\address{Department of Mathematics, University of California, San Diego, La Jolla, CA 92093-0112, USA}
\email{cpopescu@math.ucsd.edu}



\thanks{The second author was partially supported by NSF Grants DMS-0600905 and DMS-0901447}


\begin{abstract}

Let $Z\longrightarrow Z'$ be a $G$--Galois cover of smooth,
projective curves over an arbitrary algebraically closed field
$\kappa$, and let $\mathcal S$ and $\mathcal T$ be $G$--equivariant,
disjoint, finite, non-empty sets of closed points on $Z$, such that
$\mathcal S$ contains the ramification locus of the cover. In this
context, we prove that the $\ell$--adic realizations
$T_\ell(\mathcal M_{\mathcal S, \mathcal T})$ of the Picard
$1$--motive $\mathcal M_{\mathcal S, \mathcal T}$ associated to the
data $(Z, \kappa, \mathcal S, \mathcal T)$ are $G$--cohomologically
trivial and therefore $\mathbb Z_\ell[G]$--projective modules of finite
rank, for all prime numbers $\ell$. In the process, we give a new
proof of Nakajima's theorem \cite{Nakajima} on the Galois module
structure of the semi-simple piece $\Omega_Z(-[\mathcal S])^s$ under
the action of the Cartier operator on a certain space of
differentials $\Omega_Z(-[\mathcal S])$ associated to $Z$ and
$\mathcal S$, assuming that $\text{char}(\kappa)=p$. As a main
arithmetic application of these results, we consider the situation
where the set of data $(Z, Z', \mathcal S, \mathcal T)$ is defined
over a finite field $\mathbb F_q$ and $\kappa: =\overline{\mathbb F_q}$.
We combine results of Deligne and Tate, Berthelot, Bloch and Illusie
with our cohomological triviality result to prove that in this
context we have an equality of $\mathbb Z_\ell[[\overline G]]$--ideals
$\left(\Theta_{\mathcal S, \mathcal
T}(\gamma^{-1})\right)=\text{Fit}_{\mathbb Z_\ell[[\overline
G]]}(T_\ell(M_{\mathcal S, \mathcal T}))$, for all prime numbers
$\ell$, where $\overline G:=G\times\text{Gal}(\overline{\mathbb
F_q}/\mathbb F_q)$, $\gamma$ is the $q$--power arithmetic Frobenius
morphism (viewed as a distinguished topological generator of ${\rm
Gal}(\overline{\mathbb F_q}/\mathbb F_q)$) and $\Theta_{\mathcal S,
\mathcal T}(u)\in\mathbb Z[G][u]$ is the (polynomial) $G$--equivariant
$L$--function associated to the data $(Z, Z', \mathbb F_q, \mathcal S,
\mathcal T)$. We obtain this way a Galois--equivariant refinement of
results of Deligne and Tate \cite{Tate-Stark} on $\ell$--adic
realizations of Picard $1$--motives associated to (global) function
fields. As an immediate application, we prove refinements of the
classical Brumer-Stark and Coates-Sinnott conjectures linking
special values of $\Theta_{\mathcal S, \mathcal T}(u)$ to certain
invariants of ideal-class groups and \'etale cohomology groups,
respectively. In our upcoming work, we will show how several other
classical conjectures on special values of global $L$--functions
follow from the results obtained in this paper.
\end{abstract}

\maketitle

\section{\bf Introduction} The main goal of this paper is twofold.
{\it First,} we consider a Galois cover $Z\longrightarrow Z'$
of Galois group $G$ of smooth, projective curves over an arbitrary
algebraically closed field $\kappa$, and two $G$--invariant,
disjoint, finite, non-empty sets of closed points $\mathcal S$ and
$\mathcal T$ on $Z$, such that $\mathcal S$ contains the
ramification locus of the cover. In this context, we prove that the
$\ell$--adic realizations $T_\ell(\mathcal M_{\mathcal S, \mathcal
T})$ of the Picard $1$--motive $\mathcal M_{\mathcal S, \mathcal T}$
associated to the data $(Z, \kappa, \mathcal S, \mathcal T)$ are
$G$--cohomologically trivial and therefore $\mathbb
Z_\ell[G]$--projective modules of finite rank, for all prime numbers
$\ell$ (see Theorem \ref{main-ct}.)

This goal is achieved in two main steps: in the first, we give a new
interpretation of the groups $\mathcal M_{\mathcal S, \mathcal T}[n]$ of $n$--torsion
points of $\mathcal M_{\mathcal S, \mathcal T}$ (see Proposition \ref{interpret});
in the second step, we use this interpretation to give a complete description of
the $\mathbb F_{\ell}[G]$--module
structure of $\mathcal M_{\mathcal S, \mathcal T}[\ell]$ in the case
where $G$ is an $\ell$--group. We prove that this module is free and explicitly compute
its rank in terms of the genus of $Z'$ and the Hasse-Witt invariant of $Z'$,
depending on whether $\ell\ne \text{char}(\kappa)$ or $\ell=\text{char}(\kappa)$,
respectively (see Theorem \ref{ml-free}.)
This result implies our first main Theorem \ref{main-ct}.

Another immediate application of Proposition \ref{interpret} and
Theorem \ref{ml-free} is a new proof of a classical theorem of
Nakajima \cite{Nakajima} on the Galois module structure of the
semi-simple piece $\Omega_Z(-[\mathcal S])^s$ under the action of
the Cartier operator on a certain space of differentials
$\Omega_Z(-[\mathcal S])$ associated to $Z$ and $\mathcal S$,
assuming that $\text{char}(\kappa)=p$. In \S6, we show that
Nakajima's theorem is a particular case of our Theorem
\ref{ml-free}, placing this way Nakajima's results in the general
context of Picard $1$--motives (see Proposition \ref{equivalence}.)
In upcoming work \cite{Greither-Popescu2}, we will apply Proposition
\ref{interpret} to give explicit formulas for the canonical perfect
duality pairings between the torsion (or $\ell$--adic realizations)
of the Picard $1$--motive $\mathcal M_{\mathcal S, \mathcal T}$ and
its dual (Albanese) $1$--motive $\mathcal M_{T, S}$, generalizing
the classical formulas for the Weil pairings at the level of
Jacobians of curves.
\medskip

{\it Second}, we focus our attention on the arithmetically
interesting situation where $G$ is abelian, the set of data $(Z, Z',
\mathcal S, \mathcal T)$ is defined over a finite field $\mathbb F_q$
of characteristic $p$ (in the sense explained below) and
$\kappa=\mathbb F$ (the algebraic closure of $\mathbb F_q$.) In this case,
we have a $G$--Galois cover $Z_0\longrightarrow Z'_0$ of smooth
projective curves over $\mathbb F_q$, $Z=Z_0\times_{\mathbb F_q}\mathbb F$
and $Z'=Z'_0\times_{\mathbb F_q}\mathbb F$, $S$ and $T$ are two finite,
disjoint sets of closed points on $Z'_0$, $S$ contains the ramified
locus of the cover, and $\mathcal S$ and $\mathcal T$ consist of all
closed points on $Z$ sitting above points in $S$ and $T$,
respectively. To the set of data $(Z_0, Z'_0, \mathbb F_q, S, T)$ one
can associate a polynomial $G$--equivariant $L$--function
$\Theta_{S, T}(u)\in\mathbb Z[G][u]$, obtained via an equivariant
construction from the classical Artin $L$--functions for the
$G$--cover of $\mathbb F_q$--schemes $Z_0\longrightarrow Z_0'$ (see
Chpt.~V of \cite{Tate-Stark} and \S4 below.) Deligne expresses
$\Theta_{S, T}(u)$ in terms of the ($G$--equivariant) characteristic
polynomial of the $q$--power geometric Frobenius morphism acting on
the $\mathbb Q_\ell$--representation $\mathbb Q_{\ell}\otimes_{\mathbb
Z_\ell}T_{\ell}(\mathcal M_{\mathcal S, \mathcal T})$ of $G$, for
any prime $\ell\ne p$ (see \cite{Tate-Stark}, Chpt.~V.) One can
combine Deligne's result with a theorem of Berthelot and express
$\Theta_{S, T}(u)$ in terms of the $G$--equivariant $q$--power
Frobenius action on the crystalline cohomology group $H^1_{\rm
cris}(Z/\mathbb C_p)$ (see \S4 below and the Appendix of
\cite{Popescu-Stark}.) In \S4, we combine this link with a result of
Bloch and Illusie relating crystalline and $p$--adic \'etale
cohomology and with our cohomological triviality result (Theorem
\ref{main-ct}) to show that we have an equality of $\mathbb
Z_\ell[[\overline G]]$--ideals
$$\left(\Theta_{S, T}(\gamma^{-1})\right)=\text{Fit}_{\mathbb Z_\ell[[\overline G]]}(T_\ell(\mathcal M_{\mathcal S, \mathcal T})),$$
for all prime numbers $\ell$, where $\overline G:=G\times\text{\rm
Gal}(\mathbb F/\mathbb F_q)$, $\gamma$ is the $q$--power arithmetic
Frobenius viewed as a distinguished topological generator of ${\rm
Gal}(\mathbb F/\mathbb F_q)$, and ``Fit'' stands for the first Fitting
ideal (see Theorem \ref{Fitting}.) As an immediate consequence of Theorem \ref{Fitting},
in \S5 we prove refined versions
of the function field analogues
of the classical Brumer-Stark and Coates-Sinnott conjectures, linking special
values of the equivariant $L$--function $\Theta_{S, T}$ to Galois-module structure
invariants of certain ideal class-groups and \'etale cohomology groups
of $K$, respectively (see Theorems \ref{refined-Brumer-Stark} and \ref{refined-Coates-Sinnott}.)
\medskip

In upcoming work, we show how our main Theorems \ref{main-ct} and
\ref{Fitting} permit us to prove (or, in same cases, give new proofs
of) several other classical conjectures on special values of
$L$--functions for characteristic $p$ global fields (function
fields), e.g. the Rubin--Stark and Gross conjectures. Other
applications of our work on Tate modules of Picard $1$--motives
include explicit calculations of Fitting ideals over Galois--group
rings of (Pontrjagin duals of) ideal--class groups (see
\cite{Greither-Popescu}) as well as a possible new way of looking at
Tate sequences and canonical classes. Further, we are in the process
of developing a number field Iwasawa theoretic analogue of this
theory with similar results and applications, under certain
restrictive hypotheses. Just as our work over finite fields can be
viewed as a Galois-equivariant refinement of work of Deligne and
Tate \cite{Tate-Stark} on $\ell$--adic realizations of Picard
$1$--motives associated to (global) function fields, its intended
number field analogue will be a natural Galois-equivariant
refinement of Wiles's results \cite{Wiles} on the Main Conjecture in
classical Iwasawa Theory over totally real number fields.
\medskip

\noindent {\bf Acknowledgement.} The authors would like to thank
their home universities for making mutual visits possible. These visits were funded by the DFG and the NSF, respectively.
The
second author would like to thank Kiran Kedlaya for fruitful
discussions and especially for bringing Nakajima's work
\cite{Nakajima} to his attention.

\section{\bf Picard $1$--motives, their torsion points and $\ell$--adic
realizations} In this section, we recall the definitions and main
properties of a special class of $1$--motives called Picard
$1$--motives, and briefly review their torsion structure and
$\ell$--adic realizations. The section ends with a proposition which
gives a new interpretation of the group of $n$--torsion points of a
Picard $1$--motive, which will prove very useful in our future
considerations. For general properties of $1$--motives the reader
can consult Chpt. 10 of Deligne's paper \cite{Deligne-HodgeIII}. For
the arithmetic situation treated in the last section of this paper,
the reader will find Chpt.~V of \cite{Tate-Stark} useful as well.

Let $Z$ denote a smooth,
projective (not necessarily connected) curve over an algebraically
closed field $\kappa$.  Let $\mathcal
K:=\kappa(Z)$ be the $\kappa$--algebra of rational functions on
$Z$. If $Z$ is connected, then $\mathcal K$ is a field; otherwise,
$\mathcal K$ is a finite direct sum of fields (finitely generated, of transcendence degree $1$
over $\kappa$). These are the fields of rational functions on
the connected components of $Z$. We let $\mathcal S$ and
$\mathcal T$ denote two finite (possibly empty), disjoint sets of
closed points on $Z$.

\begin{definition}\label{st} For $\mathcal S$ and $\mathcal T$ as above and all $n\in\mathbb N$, we define the following
subgroups of the multiplicative group $\mathcal K^\times$ of invertible
elements in $\mathcal K.$
$$\mathcal K_{\mathcal T}^\times:=\{f\in\mathcal K^\times\mid f(P)=1,\,\forall\, P\in\mathcal T\}\,,$$
$$\mathcal K_{\mathcal S, \mathcal T}^{(n)}:=\{f\in\mathcal K_{\mathcal T}^\times\mid {\rm div}(f)=nD+y\}\,,$$
where $D$ is an arbitrary divisor on $Z$ and $y$ is a divisor on
$Z$ supported on $\mathcal S$.
\end{definition}

If $\mathcal R$ is a (possibly infinite) set of closed points on
$Z$, we denote by ${\rm Div}^0(\mathcal R)$ the abelian group of
divisors supported on $\mathcal R$ and of {\it multidegree} $0$
(i.e. of degree $0$ when restricted to each of the connected
components of $Z$.) In what follows, we will denote by $J$ the
Jacobian variety associated to $Z$. We remind the reader that $J$ is
a connected abelian variety defined over $\kappa$ of dimension equal
to the sum of the genera of the connected components of $Z$. If
$J(\kappa)$ denotes the group of $\kappa$--rational points of $J$,
then we have a well-known canonical group isomorphism
$$J(\kappa)\overset\sim\longrightarrow\frac{{\rm Div}^0(Z)}{\{{\rm div}(f)\mid
f\in\mathcal K^\times\}}\,.$$
The generalized Jacobian $J_{\mathcal T}$ associated to $(Z, \mathcal T)$
is a semi-abelian variety (an extension of the abelian variety $J$ by a torus
$\tau_{\mathcal T}$). At the level of $\kappa$--rational points, there are canonical group isomorphisms
$$\tau_{\mathcal T}(\kappa)\overset\sim\longrightarrow \bigoplus_{z\in\pi_0(Z)}\frac{\oplus_{v\in\mathcal T_z}\kappa^\times}{\kappa^\times}\,,\qquad J_{\mathcal T}(\kappa)\overset\sim\longrightarrow\frac{{\rm Div}^0(Z\setminus\mathcal T)}{\{{\rm div}(f)\mid
f\in\mathcal K_{\mathcal T}^\times\}}\,,$$ where $\pi_0(Z)$ denotes
the set of connected components of $Z$, $\mathcal T_z$ denotes the
set of those points in $\mathcal T$ which lie on $z$, and
$\kappa^\times$ sits inside $\oplus_{v\in\mathcal T_z}\kappa^\times$
diagonally, for all $z\in\pi_0(Z)$. The isomorphisms above lead to a
short exact sequence of groups
$$\xymatrix{
0\ar[r] &\tau_{\mathcal T}(\kappa)\ar[r] &J_{\mathcal T}(\kappa)\ar[r] &J(\kappa)\ar[r] &0\,,}
$$
where the right nontrivial map is the obvious one (taking the class
of a divisor $D\in {\rm Div}^0(Z\setminus\mathcal T)$ into the class
of $D$ modulo $\{{\rm div}(f)\mid f\in\mathcal K^\times\}$) and the
left nontrivial map sends the class of $(x_v)_{v\in\mathcal T}\in
\oplus_{v\in\mathcal T}\kappa^\times$ into the class of ${\rm
div}(f)$, where $f\in\mathcal K^\times$, such that $f(v)=x_v$, for
all $v\in\mathcal T$.

\begin{remark} Note that since $\tau_{\mathcal T}(\kappa)$ and $J(\kappa)$
are divisible groups, $J_{\mathcal T}(\kappa)$ is also divisible and
we have obvious exact sequences
\begin{equation}\label{semiabelian-n-torsion}0\longrightarrow
\tau_{\mathcal T}[n]\longrightarrow J_{\mathcal T}[n]\longrightarrow
J[n]\longrightarrow 0\,,\quad 0\longrightarrow
T_{\ell}(\tau_{\mathcal T})\longrightarrow T_\ell(J_{\mathcal
T})\longrightarrow T_\ell(J)\longrightarrow 0\end{equation} at the
levels of $n$--torsion points and $\ell$--adic Tate modules, for all
$n\in\mathbb N$ and all primes $\ell$.
\end{remark}
\medskip

\noindent Clearly, we have a group morphism
$${\rm Div}^0(\mathcal S)\overset{\delta}\longrightarrow J_{\mathcal T}(\kappa),$$
sending a divisor $D$ into its class modulo $\{{\rm div}(f)\mid
f\in\mathcal K_{\mathcal T}^\times\}$. (Recall that $\mathcal S\cap\mathcal T=\emptyset$.)

\begin{definition} Deligne's Picard $1$--motive $\mathcal M_{\mathcal
S, \mathcal T}$ associated to $(Z, \kappa, \mathcal S, \mathcal T)$
is, by definition, the group morphism ${\rm Div}^0(\mathcal
S)\overset{\delta}\longrightarrow J_{\mathcal T}(\kappa)$ defined
above.
\end{definition}

\begin{remark} One can obviously think of $\mathcal M_{\mathcal S,
\mathcal T}$ as associated to $(\mathcal K, \kappa, \mathcal S,
\mathcal T)$, where $\kappa$ is an algebraically closed field,
$\mathcal K$ is a semi-simple finitely generated $\kappa$--algebra,
of (pure) transcendence degree $1$ over $\kappa$, and $\mathcal S$
and $\mathcal T$ are finite, disjoint sets of valuations on
$\mathcal K$ which are trivial on $\kappa$. (In this framework, $Z$
is just a smooth, projective model for $\mathcal K$ over $\kappa$.)
\end{remark}
\medskip

Next, we recall the construction of the $\ell$--adic Tate modules
($\ell$--adic realizations) of $\mathcal M_{\mathcal S, \mathcal
T}$. For every $n\in\mathbb N$, we consider the fiber-product of groups
$J_{\mathcal T}(\kappa)\times^n_{J_{\mathcal T}(\kappa)}{\rm
Div}^0(\mathcal S)$, with respect to the map ${\rm Div}^0(\mathcal
S)\overset{\delta}\longrightarrow J_{\mathcal T}(\kappa)$ and the
multiplication by $n$ map $J_{\mathcal T}(\kappa)\overset
n\longrightarrow J_{\mathcal T}(\kappa).$ An element in this fiber
product consists of a pair $(\widehat D, x)$, where $D\in{\rm
Div}^0(Z\setminus\mathcal T)$, $\widehat D$ is the class of $D$ in
$J_{\mathcal T}(\kappa)$, $x\in{\rm Div}^0(\mathcal S)$ and
$nD-x={\rm div}(f)$, for some $f\in\mathcal K_{\mathcal T}^\times.$
Since $J_{\mathcal T}(\kappa)$ is a divisible group, we have a
commutative diagram (in the category of abelian groups) whose rows
are exact.
$$\xymatrix {
0\ar[r] & J_{\mathcal T}[n]\ar[r]\ar[d]^{=} &J_{\mathcal T}(\kappa)\times^n_{J_{\mathcal T}(\kappa)}{\rm Div}^0(\mathcal S)\ar[r]\ar[d] &{\rm Div}^0(\mathcal S)\ar[r]\ar[d]^{\delta} &0\\
0\ar[r] & J_{\mathcal T}[n]\ar[r] & J_{\mathcal T}(\kappa) \ar[r]^n & J_{\mathcal T}(\kappa)\ar[r] &0}
$$
In the diagram above, $J_{\mathcal T}[n]$ denotes the group of $n$--torsion points of $J_{\mathcal T}$.

\begin{definition} The group $\mathcal M_{S, T}[n]$ of $n$--torsion
points of $\mathcal M_{\mathcal S, \mathcal T}$ is defined to be
$$\mathcal M_{S, T}[n]:=(J_{\mathcal T}(\kappa)\times^n_{J_{\mathcal T}(\kappa)}{\rm Div}^0(\mathcal S))\otimes\mathbb Z/n\mathbb Z\,.$$
\end{definition}
\noindent Since ${\rm Div^0}(\mathcal S)$ is a free $\mathbb
Z$--module, we have commutative diagrams whose rows are exact
\begin{equation}\label{motive-n-torsion}\xymatrix {
0\ar[r] & J_{\mathcal T}[m]\ar[r]\ar@{>>}[d] &\mathcal M_{\mathcal S, \mathcal T}[m]\ar[r]\ar@{>>}[d] &{\rm Div}^0(S)\otimes\mathbb Z/m\mathbb Z\ar[r]\ar@{>>}[d] &0\\
0\ar[r] & J_{\mathcal T}[n]\ar[r] &\mathcal M_{\mathcal S, \mathcal T}[n]\ar[r] &{\rm Div}^0(S)\otimes\mathbb Z/n\mathbb Z\ar[r] &0\,,
}\end{equation}
for all $n, m\in\mathbb N$ with $n\mid m$, where the left vertical map is multiplication by $m/n$, the right vertical map is the canonical projection and the middle vertical
map is the unique morphism which makes the diagram commute.

\begin{definition} If $\ell$ is prime, then the $\ell$--adic Tate module $T_\ell(\mathcal M_{\mathcal S, \mathcal T})$ of $\mathcal M_{\mathcal S, \mathcal T}$
is given by
$$T_\ell(\mathcal M_{\mathcal S, \mathcal T})=\underset{\underset n\longleftarrow}\lim \mathcal M_{\mathcal S, \mathcal T}[\ell^n]\,,$$
where the projective limit is taken with respect to the surjective maps described in the diagram above.
\end{definition}
\noindent This way, for every prime $\ell$, we obtain exact sequences of free $\mathbb Z_\ell$--modules
$$0\longrightarrow T_{\ell}(J_{\mathcal T})\longrightarrow T_{\ell}(\mathcal M_{\mathcal S, \mathcal T})\longrightarrow{\rm Div}^0(\mathcal S)\otimes\mathbb Z_\ell\longrightarrow 0\,.$$

\begin{remark}\label{p-torsion} Let us assume that $\text{\rm char}(\kappa)=p$. In this case, if $m\in\mathbb N$, then $\tau_{\mathcal T}[p^m]=\{1\}$. (There are no non-trivial
$p$--power roots of unity in characteristic $p$.) Consequently, $J_{\mathcal T}[p^m]=J[p^m]$ and $T_p(J_\mathcal T)=T_p( J)\,.$
Via the exact sequences above, this leads to
$$\mathcal M_{\mathcal S, \mathcal T}[p^m]=\mathcal M_{\mathcal S, \emptyset}[p^m],\quad T_p(M_{\mathcal S, \mathcal T})=T_p(M_{\mathcal S, \emptyset})\,.$$
\end{remark}

\begin{remark}\label{l-infinity} Obviously, for all $n, m\in\mathbb N$, such that $n\mid m$, we have injective group morphisms
$$\mathcal M_{\mathcal S, \mathcal T}[n]\longrightarrow \mathcal M_{\mathcal S, \mathcal T}[m],$$
given by $(\widehat D, x)\otimes\hat 1\longrightarrow
(\widehat{D}, \frac{m}{n}\cdot x)\otimes\hat 1$.
Therefore, one can define $\mathcal M_{\mathcal S, \mathcal
T}[\ell^\infty]:= \underset
m{\underset{\longrightarrow}\lim}\mathcal M[\ell^m]$, for any prime
$\ell$, where the injective limit is taken with respect to the
morphisms constructed above. It is easy to see that $\mathcal
M_{\mathcal S, \mathcal T}[\ell^\infty]$ is a divisible $\mathbb
Z_{\ell}$--module of finite co-rank and there is a canonical $\mathbb
Z_{\ell}$--module isomorphism
$$T_{\ell}(\mathcal M_{\mathcal S, \mathcal T})\overset\sim\to {\rm Hom}_{\mathbb Z_\ell}(\mathbb Q_\ell/\mathbb Z_\ell,\, \mathcal M_{\mathcal S, \mathcal T}[\ell^\infty])\,.$$
\end{remark}

\begin{proposition}\label{interpret} For every $n\in\mathbb N$, we have a canonical group isomorphism
$$\mathcal K^{(n)}_{\mathcal S, \mathcal T}/\mathcal K_{\mathcal T}^{\times n}\overset\sim\longrightarrow\mathcal M_{\mathcal S, \mathcal T}[n]\,,$$
where the notations are as in Definition \ref{st}.
\end{proposition}
\begin{proof} Let us fix an $n\in\mathbb N$. An element in the fiber product $J_{\mathcal T}(\kappa)\times^n_{J_{\mathcal T}(\kappa)}{\rm Div}^0(\mathcal S)$ consists of a pair
$(\widehat D, x)$, where $D\in{\rm Div}^0(Z\setminus\mathcal T)$, $\widehat D$ is the class of $D$ in $J_{\mathcal T}(\kappa)$, $x\in{\rm Div}^0(\mathcal S)$ and $nD-x={\rm div}(f)$, for
some $f\in\mathcal K_{\mathcal T}^\times.$ First, we define a group morphism
$$J_{\mathcal T}(\kappa)\times^n_{J_{\mathcal T}(\kappa)}{\rm Div}^0(\mathcal S) \overset\phi\longrightarrow \mathcal K^{(n)}_{\mathcal S, \mathcal T}/\mathcal K_{\mathcal T}^{\times n}$$
given by $\phi(\widehat D, x)=\widehat f$, where $f\in\mathcal K_{\mathcal T}^\times$ such that ${\rm div}(f)=nD-x$, and $\widehat f$
is the class of $f$ in $\mathcal K_{\mathcal S, \mathcal T}^{(n)}/\mathcal K^{\times n}_{\mathcal T}$.
We claim that $\phi$ is a well--defined, surjective group morphism. We check this next.
\medskip

\noindent {\bf Step 1. $\phi$ is a well-defined function.}
First, let $(\widehat D, x)=(\widehat D', x)$ in $J_{\mathcal T}(\kappa)\times^n_{J_{\mathcal T}(\kappa)}{\rm Div}^0(\mathcal S)$.
The definitions show that
this means that
$$D-D'={\rm div}(g),\text{ for some } g\in\mathcal K_{\mathcal T}^\times$$
$$nD-x={\rm div}(f)\text{ and } nD'-x={\rm div}(f'),\text{ for some }f, f'\in\mathcal K_{\mathcal T}^\times\,.$$
When combined, these equalities imply that ${\rm div}(ff'^{-1})={\rm div}(g^n)$. Now, one has to work on each connected component of $Z$ separately. So, without loss of generality, we
may assume that $Z$ is connected. If $\mathcal T\ne\emptyset$, the last equality implies that
$ff'^{-1}=g^n$. If $\mathcal T=\emptyset$, then we have $\mathcal K_{\mathcal T}^\times=\mathcal K^\times$ and
$ff'^{-1}=g^n\cdot\alpha$, for some constant rational function $\alpha$ on $Z$. However, since $\kappa$ is algebraically
closed, $\alpha$ is an $n$--power in $\mathcal K^\times$ and therefore $ff'^{-1}=g'^n$, for some $g'\in\mathcal K_{\mathcal T}^\times$.
The conclusion is that in all cases we have $\widehat f=\widehat f'$ in $\mathcal K_{\mathcal S, \mathcal T}^{(n)}/\mathcal K^{\times n}_{\mathcal T}$, which concludes Step 1.

Finally, note that if $(\widehat D, x)$ and $f$ are as above, then $\widehat f$ is uniquely determined (despite the
fact that $f$ itself is not.) Indeed, the above argument shows that for $\mathcal T\ne\emptyset$ this is obvious,
while for $\mathcal T=\emptyset$ this follows from the fact that $\kappa$ is algebraically closed

\medskip

\noindent {\bf Step 2. $\phi$ is a surjective group morphism.} Only the surjectivity of $\phi$ requires a proof. This is an immediate consequence
of the equality
$$\mathcal K_{\mathcal S, \mathcal T}^{(n)}=\{f\in\mathcal K_{\mathcal T}^\times\mid {\rm div}(f)=nD-x, D\in{\rm Div}^0(Z\setminus \mathcal T),\,x\in{\rm Div}^0(\mathcal S)\}\,.$$
The inclusion of the right-hand side into the left-hand side is obvious from the definition of $\mathcal K_{\mathcal S, \mathcal T}^{(n)}$. Now, let $f\in
\mathcal K_{\mathcal S, \mathcal T}^{(n)}$ and let ${\rm div}(f)=nD+y$, where $D\in{\rm Div}(Z)$ and $y\in{\rm Div}(\mathcal S)$. If $y$ is the $0$ divisor,
then we obviously have $D\in{\rm Div}^0(Z\setminus \mathcal T)$ and we are done. (Keep in mind that ${\rm deg}({\rm div}(f))=0$ !)
If $y$ is not the $0$ divisor, let $v$ be a closed point in the support of $y$. Then, we have
$${\rm div}(f)=n(D-{\rm deg}(D)v)-(-n\cdot{\rm deg}(D)v-y)\,.$$
It is immediate that
$$D-{\rm deg}(D)v\in {\rm Div}^0(Z\setminus \mathcal T)\text{ and }(-n\cdot{\rm deg}(D)v-y)\in{\rm Div}^0(\mathcal S)\,,$$
which concludes Step 2. Therefore $\phi$ is a surjective group morphism.
\medskip

\noindent
{\bf Step 3.} Now, we claim that
$$\ker(\phi)=n (J_{\mathcal T}(\kappa)\times^n_{J_{\mathcal T}(\kappa)}{\rm Div}^0(\mathcal S))\,.$$
Indeed, let $(\widehat D, x)\in J_{\mathcal T}(\kappa)\times^n_{J_{\mathcal T}(\kappa)}{\rm Div}^0(\mathcal S)$, such that $\phi(\widehat D, x)=\widehat 0$. This is  equivalent to
$$nD-x={\rm div}(g^n)=n\cdot{\rm div}(g),$$
for some $g\in\mathcal K^\times_{\mathcal T}$. This implies that $x=nx'$, for some $x'\in {\rm Div}^0(\mathcal S)$. However, $J_{\mathcal T}(\kappa)$ is divisible, so there exists
$\widehat {D'}\in J_{\mathcal T}(\kappa)$, such that $\widehat{nD'}=\widehat D$. This means that $nD'-D={\rm div}(f')$, for some $f'\in\mathcal K_{\mathcal T}^\times$. When combined, these equalities
lead to
$$(\widehat D, x)=n(\widehat D', x'),\quad nD'-x'={\rm div}(gf')\,,$$
which proves that $\ker(\phi)\subseteq n (J_{\mathcal T}(\kappa)\times^n_{J_{\mathcal T}(\kappa)}{\rm Div}^0(\mathcal S))\,.$ The opposite inclusion is obvious.
\medskip

\noindent
Now, Steps 1--3 combined with the definition of $\mathcal M_{\mathcal S, \mathcal T}[n]$ lead to the conclusion that the group morphism $\phi$
constructed above factors through a group isomorphism $\widetilde{\phi}$
$$\xymatrix{
J_{\mathcal T}(\kappa)\times^n_{J_{\mathcal T}(\kappa)}{\rm Div}^0(\mathcal S) \ar@{>>}[r]^{\qquad \phi}\ar@{>>}[d] &\mathcal K^{(n)}_{\mathcal S, \mathcal T}/\mathcal K_{\mathcal T}^{\times n}\\
\mathcal M_{\mathcal S, \mathcal T}[n]\ar[ur]^{\widetilde\phi}_{\sim} &
}$$ This concludes the proof of the Proposition.
\end{proof}

\begin{remark}\label{G-equiv} Assume that $\kappa_0$ is a subfield of $\kappa$ and $\mathcal A$
is a group of automorphisms of $Z$ in the category of $\kappa_0$--schemes. Further, assume that
the sets of closed points $\mathcal S$ and $\mathcal T$ are
(set-wise, not necessarily point-wise) invariant under the action of
$\mathcal A$. Then it is easy to see that literally all the groups
associated to the set of data $(Z, \kappa, \mathcal S, \mathcal T)$
in the current section are endowed with a natural $\mathcal
A$--action (and can be viewed as $\mathbb Z[\mathcal A]$--modules).
Most importantly, all the group morphisms considered above become
morphisms of $\mathbb Z[\mathcal A]$--modules. In particular, the
isomorphism
$$\mathcal K^{(n)}_{\mathcal S, \mathcal T}/\mathcal K_{\mathcal T}^{\times n}\overset\sim\longrightarrow\mathcal M_{\mathcal S, \mathcal T}[n] $$
in the statement of Proposition \ref{interpret} is an isomorphism of
$\mathbb Z[\mathcal A]$--modules. A typical example is the following.
Assume that $\kappa_0$ is a subfield of $\kappa$, such that the
field extension $\kappa/\kappa_0$ is Galois of (profinite) Galois
group $\mathcal G$. Assume that $Z_0\longrightarrow Z'_0$ is a
(finite) Galois cover of smooth, projective curves over $\kappa_0$,
of (finite) Galois group $G$. Let $S'_0$ and $T'_0$ be two finite,
disjoint sets of closed points on $Z'_0$. Let
$Z:=Z_0\times_{\kappa_0}\kappa$ and
$Z':=Z'_0\times_{\kappa_0}\kappa$ be the usual extensions of scalars
to $\kappa$. Further, let $\mathcal S$ and $\mathcal T$ be the sets
of closed points on $Z$ sitting above points in $S'_0$ and $T'_0$,
respectively. Then the (profinite) group $\mathcal
A:=G\times\mathcal G$ is a group of $\kappa_0$--scheme automorphisms
of $Z$ which fixes the sets of closed points $\mathcal S$ and
$\mathcal T$ (set-wise.) In this case, it is important to note that
if one views the groups associated above to $(Z, \kappa, \mathcal S,
\mathcal T)$ as topological groups endowed with discrete topologies,
then the action of $\mathcal G$ (viewed as a topological group with
its profinite topology) on these groups is continuous. This implies
directly that $T_{\ell}(\mathcal M_{\mathcal S, \mathcal T})$
(endowed with the $\ell$--adic topology now) is a topological $\mathbb
Z_{\ell}[G][[\mathcal G]]$--module, for all primes $\ell$.
\end{remark}

\section{\bf Cohomological triviality of $\ell$--adic realizations\\ of Picard $1$--motives}

In what follows, $Z\rightarrow Z'$ will denote a (finite)
$G$--Galois cover of connected, smooth, projective curves over an
algebraically closed field $\kappa$, and $\mathcal S$ and $\mathcal
T$ are $G$-invariant, finite, non--empty, disjoint sets of closed
points on $Z$, such that $\mathcal S$ contains the ramified locus of
the cover. Let $\mathcal S'$ and $\mathcal T'$ denote the sets of
closed points on $Z'$, sitting below points in $\mathcal S$ and
$\mathcal T$, respectively. We denote by $\mathcal M$ and $\mathcal
M'$ the Picard $1$--motives associated to the sets of data $(Z,
\kappa, \mathcal S, \mathcal T)$ and $(Z', \kappa, \mathcal S',
\mathcal T')$, respectively. As usual, $\mathcal K$ and $\mathcal
K'$ denote the $\kappa$--algebras of rational functions on $Z$ and
$Z'$, respectively.

Note that we have natural inclusions $\mathcal K_{\mathcal S',
\mathcal T'}^{'(n)}\subseteq \mathcal K_{\mathcal S, \mathcal
T}^{(n)}$ and $\mathcal K_{\mathcal T'}^{'\times n}\subseteq
\mathcal K_{\mathcal T}^{\times n}$. Now, since $\mathcal K/\mathcal
K'$ is unramified at points in $\mathcal T'$, we have an equality
$\mathcal K_{\mathcal T}^\times\cap\mathcal K'=\mathcal K_{\mathcal
T'}^{'\times}$. Since $\mathcal T\ne\emptyset$, the multiplicative
group $\mathcal K_{\mathcal T}^\times$ has no torsion. This implies
right away that $\mathcal K_{\mathcal T}^{\times n}\cap \mathcal
K^{'\times}=\mathcal K_{\mathcal T'}^{'\times n}$. (Indeed, simply
take $G$--invariants in the $\mathbb Z[G]$--module isomorphism
$\mathcal K_{\mathcal T}^\times \overset\sim\to\mathcal K^{\times
n}_{\mathcal T}$ given by $x\to x^n$ and keep in mind that the
extension $\mathcal K/\mathcal K'$ is $G$--Galois.) These
observations combined with Proposition \ref{interpret} lead to
natural injective $G$--module morphisms
$$\mathcal M'[n]\hookrightarrow \mathcal M[n]\,,\qquad T_{\ell}(\mathcal M')\hookrightarrow T_{\ell}(\mathcal M)\,,$$
for all natural numbers $n$ and all prime numbers $\ell$,  where $G$ acts trivially on $\mathcal M'[n]$ and $T_{\ell}(\mathcal M')$.

\begin{theorem}\label{G-invariants}
With notations as above, we have
$$\mathcal M[n]^G=\mathcal M'[n]\,,\qquad  T_{\ell}(\mathcal M)^G=T_{\ell}(\mathcal M') $$
for all natural numbers $n$ and all prime numbers $\ell$.
\end{theorem}
\begin{proof}
Obviously, the second equality above is a consequence of the first (by passing to a projective limit.) Now, let us fix a natural number $n$.
According to Proposition \ref{interpret}, we have an exact sequence of $\mathbb Z[G]$--modules
$$\xymatrix{0\ar[r]&\mathcal K_{\mathcal T}^{\times n}\ar[r] &\mathcal K_{\mathcal S, \mathcal T}^{(n)}\ar[r] &\mathcal M[n]\ar[r] &0
}\,.$$
If we apply $G$--cohomology to the exact sequence above, we obtain an exact sequence
$$\xymatrix{0\ar[r]&\left(\mathcal K_{\mathcal T}^{\times n}\right)^G\ar[r] &\left(\mathcal K_{\mathcal S, \mathcal T}^{(n)}\right)^G\ar[r] &\mathcal M[n]^G\ar[r] &\widehat{\text{H}}^1(G, \mathcal K_{\mathcal T}^{\times n})
}\,.$$
This shows that the first equality in the Theorem follows if we show that
\begin{equation}\label{equalities}\left(\mathcal K_{\mathcal T}^{\times n}\right)^G=\mathcal K_{\mathcal T'}^{'\times n}, \quad \left(\mathcal K_{\mathcal S, \mathcal T}^{(n)}\right)^G=\mathcal K_{\mathcal S', \mathcal T'}^{'(n)}, \quad
\widehat{\text{H}}^1(G, \mathcal K_{\mathcal T}^{\times n})=0\,.\end{equation}

\noindent {\bf Step 1.} The first equality above is equivalent to
$\mathcal K_{\mathcal T}^{\times n}\cap \mathcal K'^\times=\mathcal
K_{\mathcal T'}^{'\times n}$, which was proved above.
\medskip

\noindent {\bf Step 2.} The third equality is proved as follows. First, we observe that since $\mathcal K^\times_{\mathcal T}$ has no torsion, we have an isomorphism of $\mathbb Z[G]$--modules
$\mathcal K^\times_{\mathcal T}\overset\sim\longrightarrow \mathcal K^{\times n}_{\mathcal T}$ given by the $n$--power map. Now,
we let $\mathcal K^\times_{(\mathcal T)}:=\{f\in \mathcal K\mid f(w)\ne 0, \text{ for all } w\in\mathcal T\}$. We have exact sequences of $\mathbb Z[G]$--modules
$$\xymatrix{0\ar[r]&\mathcal K^\times_{\mathcal T}\ar[r] & \mathcal K^\times_{(\mathcal T)} \ar[r]^<<<<<{\text{ev}_\mathcal T} &\bigoplus\limits_{w\in \mathcal T}\kappa(w)^\times\ar[r] &0
},\
\xymatrix{0\ar[r]&\mathcal K^\times_{(\mathcal T)}\ar[r] & \mathcal K^\times \ar[r]^<<<<<{\text{div}_\mathcal T} &\text{Div}(\mathcal T)\ar[r] &0\,,
}$$
where $\kappa(w)$ is the residue field associated to $w$ (in fact isomorphic to $\kappa$), the map $\text{ev}_{\mathcal T}$ is the $\mathcal T$--evaluation map taking $f\in \mathcal K_{(\mathcal T)}^\times$
into $(f(w))_{w\in \mathcal T}$ and $\text{div}_\mathcal T(f)$ is the piece of the divisor of $f$ supported on $\mathcal T$, for all $f\in\mathcal K^\times$. Note that the maps $\text{ev}_{\mathcal T}$ and $\text{div}_{\mathcal T}$ are surjective as a consequence of the weak approximation theorem applied to the (independent) valuations of $\mathcal K$ associated to the points in $\mathcal T$. Now, since $\mathcal K/\mathcal K'$ is unramified at points on $\mathcal T'$ and $\kappa$ is algebraically closed, we have isomorphisms of $\mathbb Z[G]$--modules
$$\bigoplus\limits_{w\in \mathcal T}\kappa(w)^\times\overset\sim\longrightarrow \mathbb Z[G]\otimes_{\mathbb Z}\bigoplus_{v\in \mathcal T'}\kappa(v)^\times,\qquad \text{Div}(\mathcal T)\overset\sim\longrightarrow \mathbb Z[G]\otimes_{\mathbb Z}\text{Div}(\mathcal T'),$$
where $G$ acts trivially on $\kappa(v)^\times$ and $\text{Div}(\mathcal T')$. This shows that $\bigoplus\limits_{w\in \mathcal T}\kappa(w)^\times$ and $\text{Div}(\mathcal T)$ are induced $G$--modules and therefore $G$--cohomologically trivial.
Consequently, we have group--isomorphisms
$$\widehat {\text{H}}^i(G, \mathcal K^\times_{\mathcal T})\overset\sim\longrightarrow \widehat {\text{H}}^i(G, \mathcal K^\times_{(\mathcal T)})\overset\sim\longrightarrow \widehat {\text{H}}^i(G, \mathcal K^\times)\,,
$$
for all $i\in\mathbb Z$. In particular, Hilbert's Theorem 90 leads to
$$\widehat {\text{H}}^1(G, \mathcal K^{\times n}_{\mathcal T})\overset\sim\longrightarrow\widehat {\text{H}}^1(G, \mathcal K^\times_{\mathcal T})\overset\sim\longrightarrow \widehat {\text{H}}^1(G, \mathcal K^\times)=0,$$
which proves the third equality in (\ref{equalities}) above.
\medskip

\noindent {\bf Step 3.}
Now, we proceed to proving the second equality in (\ref{equalities}). By the definition of $\mathcal K^{(n)}_{\mathcal S, \mathcal T}$, we have the following exact sequence of $G$--modules
$$\xymatrix{0\ar[r]&\mathcal K^{(n)}_{\mathcal S, \mathcal T}\ar[r] & \mathcal K^\times_{\mathcal T} \ar[r]^<<<<<{\delta_{\mathcal S, n}} &\text{Div}(Z\setminus \mathcal S)\otimes\mathbb Z/n\mathbb Z}\,,$$
where $\delta_{\mathcal S, n}(f)$ is the piece of the divisor of $f$ supported away from $\mathcal S$ modulo $n$, for all $f\in \mathcal K^\times_{\mathcal T}$. However, since $\mathcal K/\mathcal K'$ is unramified away from $\mathcal S$ and $\kappa$ is algebraically closed, we have an isomorphism of $G$--modules
$$\text{Div}(Z\setminus \mathcal S)\overset\sim\longrightarrow \text{Div}(Z'\setminus \mathcal S')\otimes_{\mathbb Z}\mathbb Z[G]\,.$$
Therefore, after taking $G$--invariants in the exact sequence above, we obtain an exact sequence
$$\xymatrix{0\ar[r]&\left (\mathcal K^{(n)}_{\mathcal S, \mathcal T}\right )^G\ar[r] & \mathcal K^{'\times}_{\mathcal T'} \ar[r]^<<<<<{\delta_{\mathcal S', n}} &\text{Div}(Z'\setminus \mathcal S')\otimes\mathbb Z/n\mathbb Z}\,,$$
at the level of $Z'$. This shows that $\left (\mathcal K^{(n)}_{\mathcal S, \mathcal T}\right )^G=\mathcal K_{\mathcal S', \mathcal T'}^{'(n)}$, as required.
\end{proof}

\begin{definition} We let $g_Z$ denote the genus of $Z$. If ${\rm char}(\kappa)=p$, we let $\gamma_Z$ denote the Hasse-Witt invariant of $Z$.
\end{definition}

\begin{remark}\label{genus}
Recall that if $J_Z$ is the Jacobian of $Z$, we have
$$2g_Z=\text{\rm rank}_{\mathbb Z_\ell}T_\ell(J_Z)=\text{\rm dim}_{\mathbb F_\ell} J_Z[\ell],$$
for all prime numbers $\ell\ne\text{\rm char}(\kappa)$.
Also, if $\text{\rm char}(\kappa)=p>0$, recall that
$$\gamma_Z=\text{\rm rank}_{\mathbb Z_p}T_p(J_Z)=\text{\rm dim}_{\mathbb F_p} J_Z[p]\,,\qquad g_Z\geq\gamma_Z.$$
The reader who is unfamiliar with the more standard definitions of $g_Z$ and $\gamma_Z$ should feel
free to take the above equalities as definitions.
\end{remark}
\medskip

\noindent The following well-known theorem captures the behavior of the invariants $g$ and $\gamma$ in Galois covers.

\begin{theorem}\label{Hurwitz}
Assume that $Z\to Z'$ is a (finite) $G$--Galois cover of smooth, projective, connected curves over an algebraically closed field $\kappa$ and let $\mathcal S$ be a $G$--invariant set of closed points on $Z$,
containing the ramification locus of the cover. The following hold.
\begin{enumerate}
  \item {\bf (The Hurwitz genus formula.)} If $|G|$ is coprime to $\text{\rm char}(\kappa)$, then
  $$(2g_Z-2)=|G|(2g_{Z'}-2) + \sum_{w\in \mathcal S}(e_w-1)\,,$$
  where
  $e_w$ denotes the ramification index of $w$ in the cover $Z\to Z'$.
  \item {\bf (The Deuring-Shafarevich formula.)} If $\text{\rm char}(\kappa)=p>0$ and $G$ is a $p$--group, then
  $$(\gamma_Z-1)=|G|(\gamma_{Z'}-1) + \sum_{w\in \mathcal S}(e_w-1)\,.$$
\end{enumerate}
\end{theorem}
\begin{proof} See Corollary 2.4, in Chpt.~IV of \cite{Hartshorne} for part (1) and
formula (1.1) in \cite{Nakajima} and Theorem 4.2 in \cite{Subrao} for part (2).
\end{proof}

\begin{remark}\label{rewrite-Hurwitz} With notations as above, since $\kappa$ is algebraically closed (which kills inertia, meaning that the residue field extensions
associated to closed points on $Z'$ in the cover $Z\to Z'$ are all trivial) and $\mathcal K/\mathcal K'$ is Galois, we have an equality
$$\sum_{w\in S}e_w=|G|\cdot |\mathcal S'|\,,$$
where $\mathcal S'$ is the set of closed points on $Z'$ sitting
below those in $\mathcal S$. This leads to the following equivalent
formulation of the equalities in the theorem above.
$$(2g_Z-2+ |\mathcal S|)=|G|(2g_{Z'}-2 +|\mathcal S'|)\,,\qquad (\gamma_Z-1+ |\mathcal S|)=|G|(\gamma_{Z'}-1 +|\mathcal S'|)\,.$$
\end{remark}
\medskip

\begin{theorem}\label{ml-free}
With notations as above, let $\ell$ be a prime number and assume that $G$ is an $\ell$--group. Then,  $\mathcal M[\ell]$ is a free $\mathbb F_{\ell}[G]$--module of rank $r_{\mathcal M', \ell}$, where
$$r_{\mathcal M', \ell}:=\left\{
      \begin{array}{ll}
        (2g_{Z'}-2+|\mathcal S'|+|\mathcal T'|), & \hbox{$\ell\ne\text{\rm char}(\kappa)$;} \\
        (\gamma_{Z'}-1+|\mathcal S'|), & \hbox{$\ell=\text{\rm char}(\kappa)$.}
      \end{array}
    \right.$$
\end{theorem}
\begin{proof} The main ingredients needed in the proof are Theorem \ref{G-invariants} above and the following result.
\begin{proposition}\label{G-invariant-lemma} Let $\ell$ be a prime number, $k$ a field of characteristic $\ell$, $G$ a finite $\ell$--group and $M$ a finitely
generated $k[G]$--module. If
$$\dim_k M\geq |G|\cdot \dim_k M^G,$$
then $M$ is a free $k[G]$--module of rank $r:= \dim_k M^G$.
\end{proposition}
\begin{proof} See \cite{Nakajima}, Proposition 2, \S4. \end{proof}
\noindent Now, the idea is to apply this proposition to $M:=\mathcal M[\ell]$ and $k:=\mathbb F_{\ell}$. The exact sequences (\ref{semiabelian-n-torsion}) and (\ref{motive-n-torsion})
for $n:=\ell$ give the following equality.
$$\dim_{\mathbb F_\ell}\mathcal M[\ell]=\dim_{\mathbb F_\ell}J_Z[\ell]+\dim_{\mathbb F_\ell}\tau_{\mathcal T}[\ell]+\dim_{\mathbb F_\ell}\text{Div}^0(\mathcal S)\otimes\mathbb F_{\ell}\,.$$
However, from the definitions we have
$$\dim_{\mathbb F_\ell}\text{Div}^0(\mathcal S)\otimes\mathbb F_{\ell}=|\mathcal S|-1, \qquad \dim_{\mathbb F_\ell}\tau_{\mathcal T}[\ell]=\left\{
                                                \begin{array}{ll}
                                                |\mathcal T|-1, & \hbox{$\ell\ne\text{char}(\kappa)$;} \\
                                                0, & \hbox{$\ell=\text{char}(\kappa)$.}
                                              \end{array}
                                            \right.$$
Now, we combine these equalities with Remark \ref{genus} above, to conclude that
$$\dim_{\mathbb F_\ell}\mathcal M[\ell]=r_{\mathcal M, \ell}\,,$$
for all prime numbers $\ell$. Obviously, the same equality holds for the $1$--motive $\mathcal M'$. Next, note that since the cover $Z\to Z'$ is unramified at points
in $\mathcal T'$ and $\kappa$ is algebraically closed, we have $$|\mathcal T|=|G|\cdot|\mathcal T'|\,.$$
Next, we combine the last two equalities with Remark \ref{rewrite-Hurwitz} and Theorem \ref{G-invariants} to obtain
$$\dim_{\mathbb F_\ell}\mathcal M[\ell]=|G|\cdot \dim_{\mathbb F_\ell}\mathcal M'[\ell]=|G|\cdot\dim_{\mathbb F_\ell} \mathcal M[\ell]^G\,.$$
Finally, as promised, we apply Proposition \ref{G-invariant-lemma} to $M:=\mathcal M[\ell]$ and $k:=\mathbb F_{\ell}$ to conclude that $\mathcal M[\ell]$ is
a free $F_{\ell}[G]$--module of rank $r_{\mathcal M', \ell}$.
\end{proof}
\bigskip

\noindent
Now, we are ready to pass from the study of $\mathcal M[\ell]$ to that of $T_{\ell}(\mathcal M)$. The following remark makes this possible.
\begin{remark}
Let $\ell$ be a prime number.
Since $\mathcal M[\ell^\infty]:=\underset m{\underset{\longrightarrow}\lim}\mathcal M[\ell^m]$ is a divisible $\mathbb Z_\ell$--module (see Remark \ref{l-infinity})
and $T_{\ell}(\mathcal M)$ is $\mathbb Z_\ell$--free, we have a canonical exact sequence of $\mathbb Z_\ell[G]$--modules
\begin{equation}\label{l-torsion}
0\longrightarrow T_\ell(\mathcal M)\overset{\times\ell}{\longrightarrow} T_{\ell}(\mathcal M)\longrightarrow \mathcal M[\ell]\longrightarrow 0\,.
\end{equation}
Indeed, if we apply the (exact) functor $\ast\to{\rm Hom}_{\mathbb Z_\ell}(\ast,\, \mathcal M[\ell^\infty])$ to the exact sequence
$$\xymatrix{0\ar[r]&\mathbb Z_\ell\ar[r] &\mathbb Q_\ell\ar[r] &\mathbb Q_\ell/\mathbb Z_\ell\ar[r] &0
}$$
and make the canonical identification $T_{\ell}(\mathcal M)\overset\sim\to {\rm Hom}_{\mathbb Z_\ell}(\mathbb Q_\ell/\mathbb Z_\ell,\, \mathcal M[\ell^\infty])$,
we obtain the following exact sequence
$$\xymatrix{0\ar[r]&T_{\ell}(\mathcal M)\ar[r] &{\rm Hom}_{\mathbb Z_\ell}(\mathbb Q_\ell,\, \mathcal M[\ell^\infty])\ar[r] &\mathcal M[\ell^\infty]\ar[r] &0
\,.}$$ Now, exact sequence (\ref{l-torsion}) is obtained by applying
the snake lemma to the multiplication by $\ell$--endomorphism of the
exact sequence above, and remarking that multiplication by $\ell$ is
an automorphism of the $\mathbb Q_{\ell}$--vector space ${\rm
Hom}_{\mathbb Z_\ell}(\mathbb Q_\ell,\, \mathcal M[\ell^\infty])$.
\end{remark}
\medskip

\begin{theorem}\label{connected} As above, let $\pi: Z\to Z'$ be a (finite) $G$--Galois cover of smooth, projective, connected curves defined
over an algebraically closed field $\kappa$. Let $\mathcal S$ and
$\mathcal T$ denote two finite, non-empty $G$--invariant sets of
closed points on $Z$, such that $\mathcal S$ contains the
ramification locus $\mathcal S_{\text{\rm ram}}$ of the cover. Let
$\mathcal M$ and $\mathcal M'$ be the Picard $1$--motives associated
to the data $(Z, \kappa, \mathcal S, \mathcal T)$ and $(Z', \kappa,
\mathcal S':=\pi(\mathcal S), \mathcal T':=\pi(\mathcal T))$,
respectively. Let $\ell$ be a prime. Then
\begin{enumerate}
  \item $T_\ell(\mathcal M)$ is a projective $\mathbb Z_\ell[G]$--module of finite rank;
  \item  If $G$ is an $\ell$--group, then $T_\ell(\mathcal M)$ is a free $\mathbb Z_\ell[G]$--module of rank $r_{\mathcal M', \ell}$.
\end{enumerate}
\end{theorem}

\begin{proof} Let $\ell$ be a fixed prime number. Theorem \ref{ml-free}
implies that $\mathcal M[\ell]$ is $G$--cohomologically trivial. Indeed, since $\mathcal M[\ell]$ is an $\ell$--group, it suffices to show that
$\mathcal M[\ell]$ is $G^{(\ell)}$--cohomologically trivial, where $G^{(\ell)}$ is an $\ell$--Sylow subgroup of $G$.
However, Theorem \ref{ml-free} applied to the $G^{(\ell)}$--cover $Z\to Z^{(\ell)}$, where $Z^{(\ell)}$ is a smooth projective model associated to the field $\mathcal K^{G^{(\ell)}}$ (maximal
subfield of $\mathcal K$ fixed by $G^{(\ell)}$),
implies that $\mathcal M[\ell]$ is $\mathbb F_{\ell}[G^{(\ell)}]$--free and therefore $G^{(\ell)}$--cohomologically trivial.
Now, the long exact sequence in $G'$--cohomology associated to the short exact sequence (\ref{l-torsion}) for an arbitrary subgroup
$G'$ of $G$ implies that the multiplication by $\ell$ maps give group isomorphisms
$$\widehat H^i(G', T_\ell(\mathcal M))\overset{\times \ell}{\simeq} \widehat H^i(G', T_\ell(\mathcal M))\,,$$
for all $i\in\mathbb Z$. However, since the cohomology groups $\widehat H^i(G', T_\ell(\mathcal M))$ are finitely generated (torsion)
$\mathbb Z_\ell$--modules, this implies that $\widehat H^i(G', T_\ell(\mathcal M))=0$, for all $i$ and $G'$ as above.
This shows that $T_\ell(\mathcal M)$ is $G$--cohom. trivial. Since it is $\mathbb Z_\ell$--free, $T_\ell(\mathcal M)$ is
$\mathbb Z_\ell[G]$--projective, which proves part (1).

Next, assume that $G$ is an $\ell$--group. Then the ring $\mathbb Z_{\ell}[G]$ is local. Consequently, the finitely generated projective
$\mathbb Z_{\ell}[G]$--module $T_{\ell}(\mathcal M)$ is $\mathbb Z_{\ell}[G]$--free (see Lemma 1.2 of \cite{Milnor}) of rank $r$, say. However, the exact sequence  (\ref{l-torsion})
gives an $\mathbb F_\ell[G]$--module isomorphism
$$T_{\ell}(\mathcal M)\otimes_{\mathbb Z_{\ell}[G]}\mathbb F_{\ell}[G]\simeq\mathcal M[\ell]\,.$$
Consequently, Theorem \ref{ml-free} implies that $r=r_{\mathcal M', \ell}$, concluding the proof of part (2).
\end{proof}
\bigskip

Next, we extend the theorem above to a situation of particular
arithmetic interest to us, in which the curve $Z$ is not necessarily
connected. For that purpose, we let $\kappa_0$ denote a perfect
field of arbitrary characteristic (e.g. $\kappa_0:=\mathbb F_q$, with $q$ a
power of a prime $p$) and we let $Z_0\rightarrow Z'_0$ denote a $G$--Galois
cover of smooth, projective curves over $\kappa_0$. We assume that
$Z'_0$ is geometrically connected, but $Z_0$ may not satisfy this
property. More explicitly, if we let $\kappa$ denote the algebraic
closure of $\kappa_0$, then the first of the smooth projective
$\kappa$--curves $Z':=Z'_0\times_{\kappa_0}\kappa$  and
$Z:=Z_0\times_{\kappa_0}\kappa$ is connected, but the second may not
be. This is equivalent to saying that $\kappa_0$ is algebraically
closed in the function field $\mathcal K'_0:=\kappa_0(Z'_0)$, but
may not be so inside $\mathcal K_0:=\kappa_0(Z_0)$. Let $\mathcal
K':=\mathcal K'_0\otimes_{\kappa_0}\kappa$ be the field of rational
functions on $Z'$ and $\mathcal K:=\mathcal
K_0\otimes_{\kappa_0}\kappa$ the ring of rational functions on $Z$.
Then $Z\rightarrow Z'$ is a $G$--Galois cover of smooth projective
curves over $\kappa$ and $\mathcal K$ is a $G$--Galois $\mathcal
K'$--algebra. In fact, if we let $H:=\text{Gal}(\mathcal
K_0/\mathcal K_0\cap\mathcal K')$, then $Z$ has $[G:H]$ distinct
($\kappa$--isomorphic) connected components and $\mathcal K$ is a
direct sum of as many fields which are mutually isomorphic as
$\mathcal K'$--algebras (isomorphic to a field compositum
${}_c\mathcal K:=\mathcal K_0\cdot\kappa$.) If we fix such a
compositum ${}_c\mathcal K$ (viewed as a quotient of $\mathcal K$ by
one of its maximal ideals), then ${}_c\mathcal K$ is an $H$--Galois
extension of $\mathcal K'$ and it is the function field of a smooth,
projective, connected $\kappa$--curve ${}_cZ$ which is an
$H$--Galois cover of $Z'$. (Obviously, the subscript ``$c$'' above
stands for ``connected''.)

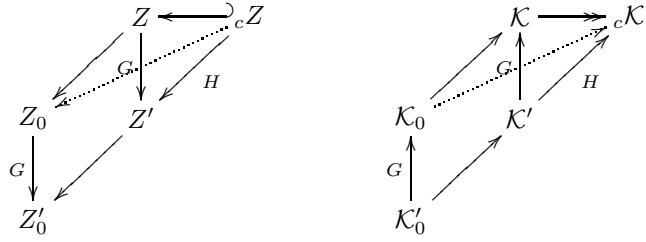
$$\xymatrix
{& Z\ar[dl]\ar[d]_{G}&{}_cZ\ar@{.>}[dll]\ar@{_(->}[l]\ar[dl]^{H}\\
Z_0\ar[d]_{G} & Z'\ar[dl]&\\
Z'_0& &}\qquad\qquad
\xymatrix
{& \mathcal K\ar@{>>}[r]&{}_c\mathcal K\\
\mathcal K_0\ar@{.>}[urr]\ar[ur] & \mathcal K'\ar[u]^{G}\ar[ur]_{H}&\\
\mathcal K'_0\ar[u]^{G}\ar[ur]& &}
$$
\medskip

\begin{theorem}\label{main-ct} Let $\pi_0: Z_0 \to Z_0'$ be a (finite) $G$--cover of smooth, projective curves defined over a perfect field $\kappa_0$.
Let $\kappa$ denote the algebraic closure of $\kappa_0$, let
$Z:=Z_0\times_{\kappa_0}\kappa$ and
$Z':=Z'_0\times_{\kappa_0}\kappa$ and $\pi:=\pi_0\times\mathbf
1_\kappa$. Assume that $Z'$ is connected. Let $S'_0$ and $T'_0$ be
finite, disjoint sets of closed points on $Z'_0$, such that $S_0$ is
non-empty and contains the ramified locus of the cover
$Z_0\longrightarrow Z'_0$. Let $\mathcal S$ and $\mathcal T$ be the
sets of closed points on $Z$ sitting above points in $S'_0$ and
$T'_0$, respectively. Let $\mathcal M$ and $\mathcal M'$ denote the
Picard $1$--motives associated to the sets of data $(Z, \kappa,
\mathcal S, \mathcal T)$ and $(Z', \kappa, \mathcal S':=\pi(\mathcal
S), \mathcal T':=\pi(\mathcal T))$, respectively. Let $\ell$ be a
prime. Then
\begin{enumerate}
  \item $T_\ell(\mathcal M)$ is a projective $\mathbb Z_\ell[G]$--module of finite rank;
  \item  If $G$ is an $\ell$--group, then $T_\ell(\mathcal M)$ is a free $\mathbb Z_\ell[G]$--module of rank $r_{\mathcal M', \ell}$.
\end{enumerate}
\end{theorem}
\begin{proof} Recall that $\mathcal K\overset\sim\rightarrow \mathcal K_0\otimes_{\mathcal K'_0}\mathcal K'$
is a (finite) direct sum of fields which are mutually isomorphic as $\mathcal K'[H]$--algebras. The $\mathcal K'[H]$--algebra morphism $\pi: \mathcal K \twoheadrightarrow {}_c\mathcal K$
induces a $\mathcal K'[H]$--algebra isomorphism between exactly one of these direct summands and ${}_c\mathcal K$.
The inverse of this isomorphism produces a $\mathbb Z[H]$--linear section $\iota: {}_c\mathcal K^\times \hookrightarrow \mathcal K^\times$
of the $\mathbb Z[H]$--module morphism $\mathcal K^\times \twoheadrightarrow {}_c\mathcal K^\times$ induced by $\pi$. It is easy to check that this gives
an isomorphism of $\mathbb Z[G]$--modules
$$
\iota\otimes\mathbf{1}: {}_c\mathcal K^\times\otimes_{\mathbb Z[H]}\mathbb Z[G]\overset{\sim}\longrightarrow \mathcal K^\times\,.$$
Now, let $\mathcal S_c$ and $\mathcal T_c$ denote the sets of closed points on ${}_cZ$, sitting above points in $S_0'$ and $T_0'$, respectively.
Then $\iota\otimes\mathbf{1}$ induces isomorphisms of $\mathbb Z[G]$--modules
$$
\iota\otimes\mathbf{1}: {}_c\mathcal K^{(n)}_{\mathcal S_c, \mathcal T_c}\otimes_{\mathbb Z[H]}\mathbb Z[G]\overset{\sim}\longrightarrow \mathcal K^{(n)}_{\mathcal S, \mathcal T}\,,\qquad
\iota\otimes\mathbf{1}: {}_c\mathcal K^{\times n}_{\mathcal T_c}\otimes_{\mathbb Z[H]}\mathbb Z[G]\overset{\sim}\longrightarrow \mathcal K^{\times n}_{\mathcal T}\,,$$
for all $n\in\mathbb N$ (se Definition \ref{st}.) Consequently, if we denote by $\mathcal M_c$ the Picard $1$--motive associated to data $({}_cZ, \kappa,  \mathcal S_c, \mathcal T_c)$, then we have isomorphisms
$$\mathcal M_c[n]\otimes_{\mathbb Z[H]}\mathbb Z[G]\overset{\sim}\longrightarrow\mathcal M[n]\,,\qquad T_{\ell}(\mathcal M_c)\otimes_{\mathbb Z_\ell[H]}\mathbb Z_\ell[G] \overset{\sim}\longrightarrow T_{\ell}(\mathcal M)\,,$$
of $\mathbb Z/n\mathbb Z[G]$ and $\mathbb Z_\ell[G]$--modules, respectively, for all $n\in\mathbb N$ and all prime numbers $\ell$. Now, since $Z_c$ is connected, Theorem \ref{connected} implies that $T_\ell(\mathcal M_c)$ is $\mathbb Z_\ell[H]$--projective of finite rank and free of rank $r_{\mathcal M', \ell}$, if $G$ is an $\ell$--group. Both parts (1) and (2) of our Theorem are immediate consequences of these facts combined with the second isomorphism above. \end{proof}

\section{\bf Fitting ideals of $\ell$--adic realizations of Picard \\ $1$--motives defined over finite fields}
 In this section, we apply Theorem \ref{main-ct} above in the arithmetically interesting situation where
 $\kappa_0$ is a finite field ($\kappa_0=\mathbb F_q$, where $q$ is a power of a prime $p$) and $G$ is abelian. In this case,
  $\mathcal K_0/\mathcal K'_0$ is
 an abelian $G$--Galois extension of characteristic $p$ global fields (see the set-up and notations leading into Theorem \ref{main-ct}.)
Our goal is to compute the first Fitting ideal
 of $T_\ell(\mathcal M)$ over a certain profinite $\ell$--adic group ring in terms of special values of $G$--equivariant
 $L$--functions, for all prime numbers $\ell$. \medskip

\subsection{\bf Algebraic preliminaries.}
Let $R$ be a commutative ring with $1$, $P$ a finitely generated, projective $R$--module and $f\in{\rm End}_R(P)$. Then,
the determinant $\det_R(f\mid P)$ of $f$ acting on $P$ makes sense and it is defined as follows. We take a finitely generated $R$--module $Q$,
such that $P\oplus Q$ is a (finitely generated) free $R$--module, then we let $f\oplus\mathbf 1_Q\in{\rm End}_R(P\oplus Q)$, where
$\mathbf 1_Q$ is the identity of $Q$, and define
$${\rm det}_R(f\mid P):={\rm det}_R(f\oplus \mathbf 1_Q\mid P\oplus Q)\,.$$
It is easy to check (Schanuel's Lemma !) that the definition above does not depend on $Q$. Now, one can use the same strategy to define
the characteristic polynomial ${\rm det}_R(1-f\cdot X\mid P)\in R[X]$ of variable $X$. Indeed, $P\otimes_R R[X]$ is a finitely
generated, projective $R[X]$--module. One defines
$${\rm det}_R(1-f\cdot X\mid P):={\rm det}_{R[X]}(1-f\otimes X\mid P\otimes_R R[X])\,.$$
Obviously, for any $P$, $R$ and $f$ as above and any $R$--algebra $R'$, we have base-change equalities
\begin{equation}\label{base-change}
\begin{array}{c}
  {\rm det}_R(f\mid P)={\rm det}_{R'}(f\otimes\mathbf 1_{R'}\mid P\otimes_R R')\,, \\
{}\\
  {\rm det}_R(1-f\cdot X\mid P)={\rm det}_{R'}(1- (f\otimes\mathbf 1_{R'})\cdot X\mid P\otimes_R R')\,.
\end{array}
\end{equation}
\noindent Now, for any $R$ as above and any finitely presented $R$--module $M$, the first Fitting invariant (ideal) ${\rm Fit}_R(M)$ of $M$ over $R$ is defined as follows.
First, one considers a finite presentation of $M$
$$\xymatrix{
R^n\ar[r]^\phi &R^m\ar[r] &M\ar[r] &0}\,.$$ By definition, the
Fitting ideal ${\rm Fit}_R(M)$ is the ideal in $R$ generated by the
determinants of all the $m\times m$ minors of the matrix $A_\phi$
associated to $\phi$ with respect to two $R$--bases of $R^n$ and
$R^m$. It is well-known that the definition does not depend on the
chosen presentation or bases, and
$${\rm Ann}_R(M)^m\subseteq{\rm Fit}_R(M)\subseteq {\rm Ann}_R(M)\,.$$
For more on Fitting ideals, the reader can consult the first sections of \cite{Popescu-Stark} and the references therein.
\medskip

Next, we state and prove a proposition which is of independent interest and will play an important role in our Fitting ideal calculations. In what follows, if $R$ is a commutative topological ring and $\mathcal G$ is
a commutative profinite group, then the profinite group algebra
$$R[[\mathcal G]]:=\underset{\mathcal H}{\underset\longleftarrow\lim} R[\mathcal G/\mathcal H],$$
where $\mathcal G/\mathcal H$ are all the finite quotients of
$\mathcal G$ by closed subgroups $\mathcal H$, is viewed as a
topological $R$--algebra endowed with the usual projective limit
topology.
\medskip
\begin{proposition}\label{fitting-calculation}
Let $R$ be a semi-local compact topological ring and $\mathcal G$ a pro-cyclic group of topological generator $g$. Let $M$ be a topological $R[[\mathcal G]]$--module,
which is projective and finitely generated as an $R$--module. Let $F(u):={\rm det}_{R}(1-\mu_g\cdot u\mid M)$, where $\mu_g$ is the $R[[\mathcal G]]$--automorphism of $M$ given
by multiplication by $g$. Then, $M$ is finitely presented as an $R[[\mathcal G]]$--module and we have an equality of $R[[\mathcal G]]$-ideals
$${\rm Fit}_{R[[\mathcal G]]}(M)=(F(g^{-1}))\,.$$
\end{proposition}
\begin{proof} Since $R$ is semi-local and Fitting ideals and characteristic polynomials are well behaved with respect to base-change,
we may work on each local component separately. Consequently, we may assume that $R$ is local and $M$ is $R$--free of finite
rank. Let $\{x_1, \dots, x_n\}$ be an $R$--basis for $M$ and let $A_g\in{\rm GL}_n(R)$ be the matrix associated to $\mu_g$ in this basis. We denote by $\phi_g$ the $R[[\mathcal G]]$--linear
endomorphism of $R[[\mathcal G]]^n$ whose matrix in the standard $R[[\mathcal G]]$--basis $\{e_1, \dots, e_n\}$ is $(1-g^{-1}\cdot A_g)\in M_n(R[[\mathcal G]])$. We claim that we have
an exact sequence of topological $R[[\mathcal G]]$--modules
\begin{equation}\label{presentation}\xymatrix{
R[[\mathcal G]]^n\ar[r]^{\phi_g} & R[[\mathcal G]]^n\ar[r]^{\quad \pi} &M\ar[r] &0,
}\end{equation}
where $\pi$ is the $R[[\mathcal G]]$-linear morphism satisfying $\pi(e_i)=x_i$, for all $i=1, \dots, n$. By definition, $\pi$ is surjective and $\pi\circ\phi_g=0$. So, the only thing that needs
to be checked is the inclusion $\ker\pi\subseteq {\rm Im }\,\phi_g$. In order to prove this, let $\iota: R^n\longrightarrow R[[\mathcal G]]^n$ be the canonical ($R$--linear) inclusion. Note that if $\{\overline e_1,\dots, \overline e_n\}$ is the standard $R$--basis of $R^n$, then $\iota(\overline e_i)=e_i$, for all $i$. Since $M$ is
$R$--free of basis $\{x_1, \dots, x_n\}$, the compositum $\pi\circ\iota$ is an isomorphism of $R$--modules. Consequently, we have ${\rm Im}\,\iota\,\cap \ker\pi =0$. Further, we claim
that we have an equality
\begin{equation}\label{equality}{\rm Im}\,\iota+ {\rm Im}\,\phi_g =R[[\mathcal G]]^n.\end{equation}
In order to prove this last claim, for all $m\in\mathbb N$ we let $\phi_{g, m}$ be the $R[[\mathcal G]]$--module endomorphism of $R[[\mathcal G]]^n$ whose matrix in the basis
$(e_i)_i$ is $(g^m\cdot I_n - A_g^m)\in  M_n(R[[\mathcal G]])$. It is easily seen that for all $m$ we have $\phi_{g, m}=\phi_g\circ\alpha_{g, m}$, for an easily computable
$\alpha_{g, m}\in{\rm End}_{R[[\mathcal G]]}(R[[\mathcal G]]^n)$. This shows that ${\rm Im}\, \phi_{g, m}\subseteq {\rm Im}\, \phi_g$, for all $m$. Since, $A_g$ has entries in $R$, this implies further that
$g^m\cdot e_i\in{\rm Im}\, \phi_g +{\rm Im}\, \iota$, for all $m$ and all $i$. Consequently, we have an inclusion
$$\bigoplus\limits_{i=1}^n R[\mathcal G]\cdot e_i\subseteq {\rm Im}\, \phi_g + {\rm Im}\, \iota\,.$$
Now, we take the topological closure in $R[[\mathcal G]]^n$ of both sides of the above inclusion. Since $R^n$ and $R[[\mathcal G]]^n$ are compact and $\iota$ and $\phi_g$ are continuous, the right--hand side is already closed.
Moreover, in the profinite topology, $R[\mathcal G]$ is dense in $R[[\mathcal G]]$. Consequently, we have
$$R[[\mathcal G]]^n=\bigoplus\limits_{i=1}^n R[[\mathcal G]]\cdot e_i\subseteq {\rm Im}\, \phi_g + {\rm Im}\, \iota\,,$$
which concludes the proof of our last claim (\ref{equality}).
Now, the desired inclusion $\ker\pi\subseteq {\rm Im }\,\phi_g$ is easily obtained by combining
$${\rm Im}\,\iota + {\rm Im}\,\phi_g =R[[\mathcal G]]^n,\quad {\rm Im}\,\phi_g\subseteq \ker\pi,\text{ and } {\rm Im}\,\iota\,\cap\ker \pi=0.$$
This concludes the proof of the fact that (\ref{presentation}) is indeed an exact sequence of $R[[\mathcal G]]$--modules. On one hand, this implies that
$M$ is indeed $R[[\mathcal G]]$--finitely presented. On the other hand, from the definition of the Fitting ideal, we have an equality of $R[[\mathcal G]]$--ideals
$${\rm Fit}_{R[[\mathcal G]]}(M)=({\rm det}(1-g^{-1}\cdot A_g))=(F(g^{-1}))\,.$$
This concludes the proof of our Proposition.
\end{proof}
\begin{corollary}\label{fitting-duals} Let $R$, $\mathcal G$ and $M$ be as in Proposition
\ref{fitting-calculation}. The following hold.
\begin{enumerate}\item
Let $M^\ast_R:={\rm Hom}_R(M, R)$. We view $M^\ast_R$ as a
topological $R[[\mathcal G]]$--module with the co-variant
$\mathcal G$--action, given by $\sigma\cdot
f(m):=f(\sigma\cdot m)$, for all $f\in M^\ast_R$,
$\sigma\in\mathcal G$ and $m\in M$. Then, we have
$${\rm Fit}_{R[[\mathcal G]]}(M)={\rm Fit}_{R[[\mathcal G]]}(M^\ast_R)\,.$$
\item Further, assume  $R=\mathbb Z_{\ell}[G]$, where $G$ is
    a finite, abelian group. Let $M^\ast:={\rm Hom}_{\mathbb
    Z_\ell}(M, \mathbb Z_\ell)$, viewed as an $R[[\mathcal
    G]]$--module with the co-variant $G\times\mathcal
    G$--action. Then
    $${\rm Fit}_{R[[\mathcal G]]}(M^\ast)={\rm Fit}_{R[[\mathcal
    G]]}(M).$$
\end{enumerate}
\end{corollary}
\begin{proof} (1) As in the proof of Proposition \ref{fitting-calculation}, we may assume that $M$ is a free $R$--module of basis $\overline x:=\{x_1, \dots, x_n\}$. Then $M^\ast_R$ is
$R$--free of basis $\overline x^\ast:=\{x_1^\ast, \dots,
x_n^\ast\}$, uniquely characterized by $x_i^\ast(x_j)=\delta_{i,
j}$, for all $i, j=1, \dots, n$. If one denotes by $\mu_g^\ast$
the multiplication by $g$ map on $M^\ast_R$ and by $A_g^\ast$ its
matrix with respect to the basis $\overline x^\ast$, it is easy to
see that, under the co-variant $\mathcal G$--action, we have
$A_g^\ast=A_g^t$, where $A_g^t$ is the transposed of $A_g$.
Consequently, we have
$$F^\ast(u):={\rm det}_{R}(1-\mu_g^\ast\cdot u\mid M^\ast_R)={\rm det}_{R}(1-\mu_g\cdot u\mid M)=F(u)\,.$$
Now, part (1) of the corollary follows by applying Proposition
\ref{fitting-calculation} to $M$ and $M^\ast_R$.

Part (2) follows immediately from (1) and the isomorphism of
$R[[\mathcal G]]$--modules $M^\ast\overset\sim\longrightarrow
M_R^\ast$ sending $\phi\in M^\ast$ to $\psi\in M_R^\ast$, with
$\psi(x)=\sum_{\sigma\in G}\phi(x^{\sigma^{-1}})\cdot\sigma$, for
all $x\in M$.
\end{proof}
\medskip

\subsection{\bf The statement of the problem.}
Now, we are ready to state precisely the number theoretic task
described in the introduction to this section. For that purpose, we
make a slight change in notations, in order to be in tune with the
more number theoretically minded reader. We let $K/k$ denote an
abelian $G$--Galois extension of characteristic $p$ global fields.
Assume that $\mathbb F_q$ is the exact field of constants in $k$ (but
not necessarily in $K$.) Let $X\mapsto Y$ be the corresponding
$G$-Galois cover of smooth projective curves defined over $\mathbb
F_q$. Let $S$ and $\Sigma$ be two finite, nonempty, disjoint sets of
closed points on $Y$, such that $S$ contains the set $S_{\text ram}$
of points which ramify in $X$. We let $\mathbb F$ denote the algebraic
closure of $\mathbb F_q$, $\overline X:=X\times_{\mathbb F_q}\mathbb F$,
$\overline X:=X\times_{\mathbb F_q}\mathbb F$. Also, $\overline S$ and
$\overline \Sigma$ denote the sets of closed points on $\overline X$
sitting above points in $S$ and $T$, respectively.  As in the last
section, we let $H:=\text{Gal}(K/K\cap k_\infty)$, where $k_\infty:=k\mathbb F$. For every closed
point $v$ on $Y$ which does not belong to $S_{\rm ram}$, we denote
by $G_v$ and $\sigma_v$ the decomposition group and Frobenius
automorphism associated to $v$ inside $G$, respectively. Also, for a
closed point $v$ of $Y$, we let $d_v$ denote its residual degree
over $\mathbb F_q$ (i.e. the degree over $\mathbb F_q$ of the residue
field associated to $v$) and we let $Nv:=q^{d_v}=|\mathbb F_{q^{d_v}}|$
(i.e. the cardinality of the residue field associated to $v$.)
\medskip

To the set of data $(K/k, \mathbb F_q, S, T)$, one can associate a
(polynomial) equivariant $L$--function
\begin{equation}\label{product-formula}
\Theta_{S, \Sigma}(u):=\prod\limits_{v\in \Sigma}(1-\sigma_v^{-1}\cdot(qu)^{d_v})\cdot\prod\limits_{v\not\in S}(1-\sigma_v^{-1}\cdot u^{d_v})^{-1}\,,
\end{equation}
where the infinite product on the right is taken over all closed
points in $Y$ which are not in $S$. The product on the right-hand
side is convergent in $\mathbb Z[G][[u]]$ and in fact it converges to
an element $\Theta_{S,\Sigma}(u)\in\mathbb Z[G][u]$ (see
\cite{Tate-Stark}, Chpt.~V.) The link between $\Theta_{S,\Sigma}(u)$
and the classical Artin $L$--functions associated to the characters
(irreducible representations) of the Galois group $G$ is as follows.
For every complex valued (irreducible) character $\chi$ of $G$, we
let $L_{S, \Sigma}(\chi, s)$ denote the $(S,\Sigma)$--modified Artin
$L$--function associated to $\chi$. This is the (unique) holomorphic
function of complex variable $s$, satisfying the equality
$$L_{S, \Sigma}(\chi, s)=\prod\limits_{v\in \Sigma}(1-\chi(\sigma_v)\cdot Nv^{1-s})\cdot\prod\limits_{v\not\in S}(1-\chi(\sigma_v)\cdot Nv^{-s})^{-1}\,,\quad \text{ for all }s\in\mathbb C, \text{ with }\Re(s)>1\,.$$
It is not difficult to show that we have an equality
$$\Theta_{S, \Sigma}(q^{-s})=\sum_{\chi}L_{S, \Sigma}(\chi, s)\cdot e_{\chi^{-1}}\,,\text{ for all }s\in\mathbb C\,,$$
where the sum is taken with respect to all the irreducible complex valued characters $\chi$ of $G$ and $e_\chi:=1/|G|\sum_{\sigma\in G}\chi(\sigma)\sigma^{-1}$ denotes
the idempotent corresponding to $\chi$ in $\mathbb C[G]$.
\medskip

We denote by $M_{\overline S,\overline \Sigma}$ the Picard
$1$-motive associated to the set of data $(\overline X, \mathbb F,
\overline S, \overline \Sigma)$. For a prime number $\ell$ we
consider the $\ell$--adic Tate module ($\ell$--adic realization)
$T_\ell(M_{\overline S, \overline\Sigma})$ of $M_{\overline S,
\overline \Sigma}$, endowed with the usual $\mathbb Z_\ell[[\overline
G]]$--module structure, where $\overline G:=G\times\Gamma$ and
$\Gamma:=\text{Gal}(\mathbb F/\mathbb F_q)$ (see Remark \ref{G-equiv}.) We
denote by $\gamma$ the $q$--power arithmetic Frobenius, which is the
distinguished topological generator of $\Gamma$ characterized by
$\gamma(\zeta)=\zeta^q$, for all $\zeta\in\mathbb F$. Theorem
\ref{main-ct} restricted to this context assures us that
$T_\ell(M_{\overline S, \overline\Sigma})$ is $\mathbb
Z_\ell[G]$--projective of finite rank. Proposition
\ref{fitting-calculation} combined with Remark \ref{G-equiv} implies
that $T_\ell(M_{\overline S, \overline\Sigma})$ is finitely
presented as a (topological) $\mathbb Z_\ell[[\overline G]]$--module,
for all prime numbers $\ell$. The main goal of this section is to
prove part (2) of the following.
\begin{theorem}\label{Fitting}
Under the above hypotheses, the following hold for all prime numbers $\ell$.
\begin{enumerate}
\item The $\mathbb Z_\ell[G]$--module $T_\ell(M_{\overline S, \overline\Sigma})$ is projective.
\item We have an equality of $\mathbb Z_\ell[[\overline G]]$--ideals
$$\left(\Theta_{S, \Sigma}(\gamma^{-1})\right)=\text{\rm Fit}_{\mathbb Z_\ell[[\overline G]]}(T_\ell(M_{\overline S, \overline\Sigma})).$$
\end{enumerate}
\end{theorem}
\noindent Of course, only (2) requires a proof. This task will be accomplished in two steps, the first
(and the easier) step dealing with primes $\ell\ne p$, and the
second dealing with the characteristic prime $p$.
\medskip

\subsection{\bf The case $\ell\ne p$.} We work under the assumptions and with the notations introduced in the previous subsection. The following theorem provides the link between the $\ell$--adic
realization $T_\ell(M_{\overline S, \overline\Sigma})$ of the Picard $1$--motive $M_{\overline S, \overline \Sigma}$ and the $G$--equivariant $L$--function $\Theta_{S,\Sigma}(u)$.

\begin{theorem}\label{Deligne}{\bf{\rm (Deligne)}} For all primes $\ell\ne p$, the following equality holds in $\mathbb Q_\ell[G]$.
$$\Theta_{S, \Sigma}(u)={\rm det}_{\mathbb Q_\ell[G]}(1-\gamma\cdot u\mid T_{\ell}(M_{\overline S, \overline\Sigma})\otimes_{\mathbb Z_\ell}\mathbb Q_\ell)\,.$$
\end{theorem}
\begin{proof} See Chpt.~V of \cite{Tate-Stark}.
\end{proof}
Now, we are ready to state and prove the main result of this subsection.
\begin{theorem}\label{Fitting-ell}
Under the above hypotheses, the following hold for all primes $\ell\ne p$.
\begin{enumerate}
\item The $\mathbb Z_\ell[G]$--module $T_\ell(M_{\overline S, \overline\Sigma})$ is projective.
\item
We have the following equality of ideals in $\mathbb Z_\ell[[\overline G]]$.
$$\left(\Theta_{S, \Sigma}(\gamma^{-1})\right)=\text{\rm Fit}_{\mathbb Z_\ell[[\overline G]]}(T_\ell(M_{\overline S, \overline\Sigma}))\,.$$
\end{enumerate}
\end{theorem}
\begin{proof} Obviously, part (1) above is a particular case of Theorem \ref{main-ct}. (Set $\kappa_0:=\mathbb F_q$ and $\kappa:=\mathbb F$.)
Now, part (1) above and the base-change equalities (\ref{base-change}) permit us to base-change from $\mathbb Q_\ell[G]$ down to $\mathbb Z_\ell[G]$ in the statement of Theorem \ref{Deligne}, which leads to an equality
$$\Theta_{S, \Sigma}(u)={\rm det}_{\mathbb Z_\ell[G]}(1-\gamma\cdot u\mid T_{\ell}(M_{\overline S, \overline\Sigma}))\,.$$
Finally, part (2) of the statement is a direct consequence of the equality above and Proposition \ref{fitting-calculation} applied to $M:=T_\ell(M_{\overline S, \overline\Sigma})$,
$R:=\mathbb Z_\ell[G]$, $\mathcal G:=\Gamma$, and $g:=\gamma$. Note that $R$ is indeed compact in the $\ell$--adic topology and semi-local (of local direct summands $\mathbb Z_\ell[\chi][L]$, where $\chi$ runs through a complete
set of representatives for $G(\overline{\mathbb Q_\ell}/\mathbb Q_\ell)$--conjugacy classes of $\overline{\mathbb Q_\ell}$--valued irreducible characters of $G$, and $L$ is the $\ell$--Sylow subgroup of $G$.)
\end{proof}

\subsection{\bf The case $\ell=p$. } Compared to the case $\ell\ne p$, the main difficulty in this case stems from the fact that the equality in Theorem \ref{Deligne} does not
hold for $\ell=p$. Roughly speaking, this is due to the fact that
$\text{rank}_{\mathbb Z_p}T_p(M_{\overline S, \overline \Sigma})$ is
strictly smaller than $\text{rank}_{\mathbb Z_\ell}T_\ell(M_{\overline
S, \overline \Sigma})$, for $\ell\ne p$, as follows immediately from
Theorem \ref{ml-free} above. A more enlightening way of saying this
is that the $\mathbb Z_p$--rank of $p$--adic \'etale cohomology in
characteristic $p$ is smaller than the $\mathbb Z_\ell$--rank of
$\ell$--adic \'etale cohomology, for all $\ell\ne p$. (In fact, the
modules $T_\ell(M_{\overline S, \overline \Sigma})$ can be
identified with the functorial duals of the first $\ell$--adic
\'etale cohomology groups of $M_{\overline S, \overline \Sigma}$.)
This is why, in order to obtain a $p$--adic equality similar to the
one in Theorem \ref{Deligne}, one has to replace $p$--adic \'etale
cohomology with something larger (of the right rank), namely
crystalline cohomology. We describe this next.
\medskip

If $W:=W(\mathbb F)$ denotes the ring of Witt vectors associated to
$\mathbb F$, then we let $H^1_{\rm cris}:=H^1_{\rm cris}(\overline
X/W)$  denote the first crystalline cohomology group associated to
$\overline X$. We denote by $H^1_p:=H^1_{\text{\'et}}(\overline X,
\mathbb Z_p)$ the first $p$--adic \'etale cohomology group associated
to $\overline X$. For the properties of crystalline and \'etale
cohomology relevant in this context, the reader can consult the
Appendix of \cite{Popescu-Stark} and the references therein. In
particular, we remind the reader that these are free, finite rank
$W$--modules and $\mathbb Z_p$--modules, respectively, on which the
Galois group $G$ acts naturally. The $q$--power geometric Frobenius
endomorphism $F$ associated to $\overline X$ induces a
$W[G]$--module endomorphism and a $\mathbb Z_p[G]$--module
endomorphisms (both denoted by $F^\ast$) of  $H^1_{\rm cris}$ and
$H^1_p$, respectively. An important theorem of Bloch and Illusie
(see II.5.4 in \cite{Illusie-deRham}, Lemma 3.3 in \cite{Crew-Katz}, or Appendix of \cite{Popescu-Stark}) identifies $H^1_p$ with a $\mathbb
Z_p[G]$--submodule of  $H^1_{\rm cris}$, such that
\begin{equation}\label{et-cris} H^1_p\otimes_{\mathbb Z_p}\mathbb C_p = \left(H^1_{\rm cris}\otimes_{W}\mathbb C_p\right)_0\,,\end{equation}
where $\left(H^1_{\rm cris}\otimes_{W}\mathbb C_p\right)_0$ is the $\mathbb C_p[G]$--submodule
of
$H^1_{\rm cris}\otimes_{W}\mathbb C_p$ on which $F^\ast$ acts with eigenvalues which are $p$--adic units (i.e. of $p$--adic valuation equal to $0$.)
We have a direct sum decomposition
\begin{equation}\label{dec-cris}
H^1_{\rm cris}\otimes_{W}\mathbb C_p=\left(H^1_{\rm cris}\otimes_{W}\mathbb C_p\right)_0\oplus \left(H^1_{\rm cris}\otimes_{W}\mathbb C_p\right)_{>0}
\end{equation}
in the category of $\mathbb C_p[G]$--modules, where $\left(H^1_{\rm cris}\otimes_{W}\mathbb C_p\right)_{>0}$ is the $\mathbb C_p[G]$--submodule of $H^1_{\rm cris}\otimes_{W}\mathbb C_p$ on which $F^\ast$ acts with eigenvalues of strictly positive $p$--adic valuation.
\medskip

When dealing with the $G$--equivariant (polynomial) $L$--function $\Theta_{S, \Sigma}(u)$, it is more convenient to work with the functorial duals $H_{1, \rm cris}$ and $H_{1,p}$
of $H^1_{\rm cris}$ and $H^1_p$, respectively (i.e. the homology groups.) The (contravariant) action of the geometric Frobenius $F^\ast$ on the cohomology groups turns into the (covariant)
action of the arithmetic Frobenius $F_\ast:=\gamma$ on the homology groups. Of course, we have an equality and direct sum decomposition similar to (\ref{et-cris}) and (\ref{dec-cris}),
respectively, with homology replacing co-homology and the subscripts ``$0$'' and ``$>0$'' referring to the $p$--adic valuations of the eigenvalues of $\gamma$ acting on homology. We remind the reader that there
is a canonical isomorphism ($\mathbb Z_p[G]$--linear and preserving the $\gamma$--action on both sides)
\begin{equation}\label{etale-jacobian}
H_{1,p}\overset\sim\to T_p(J_{\overline X})\,,
\end{equation}
where $T_p(J_{\overline X})$ denotes the $p$--adic Tate module of the Jacobian $J_{\overline X}$ associated to $\overline X$ (see \cite{Popescu-Stark}.) The $p$--adic homological interpretation of
$\Theta_{S, \Sigma}(u)$ is given by the following theorem, essentially due to Berthelot (see the Appendix of \cite{Popescu-Stark} and \S2.5 of \cite{Illusie-cohcris}.)

\begin{theorem}\label{theta-poly}  The following equality holds true in $\mathbb C_p[G][u]$ ($\det$ always taken over $\mathbb C_p[G]$).
$$
\Theta_{S, \Sigma}(u)=\frac{\prod_{v\in \Sigma}(1-\sigma_v^{-1}\cdot(qu)^{d_v})}{\det(1-\gamma\cdot qu\mid \mathbb C_p[G/H])}\cdot  \det(1-\gamma\cdot u\mid H_{1, \rm cris}\otimes_W\mathbb C_p)\cdot \det(1-\gamma u\mid {\rm Div}^0(\overline S)\otimes\mathbb C_p)\,,
$$
where $\gamma$ acts on $\mathbb C_p[G/H]$ by multiplication with the inverse of the canonical generator $\overline\gamma$ of $G/H$.
\end{theorem}

\begin{remark}\label{sigma}
Note that if we let $\pi:G\twoheadrightarrow G/H$ denote the canonical projection, then $\pi(\sigma_v)=\overline\gamma^{d_v}$.
This implies right away that we have a divisibility in $\mathbb C_p[G][u]$
$${\rm det}_{\mathbb C_p[G]}(1-\gamma\cdot qu\mid \mathbb C_p[G/H])\mid (1-\sigma_v^{-1}\cdot(qu)^{d_v})\,,$$
for all $v\in\Sigma$. Consequently, we have
$$\frac{\prod_{v\in \Sigma}(1-\sigma_v^{-1}\cdot(qu)^{d_v})}{\det_{\mathbb C_p[G]}(1-\gamma\cdot qu\mid \mathbb C_p[G/H])}\in\mathbb C_p[G][u]\,.$$
\end{remark}
\bigskip

Next, we express the right-hand side of the equality in the Theorem above as a product
of two polynomials $P, Q\in\mathbb C_p[G][u]$, which we define below.

$$\begin{array}{ll}
P(u)&:={\rm det}_{\mathbb C_p[G]}(1-\gamma\cdot u\mid \left(H_{1, \rm cris}\otimes_W\mathbb C_p\right)_0)\cdot{\rm det}_{\mathbb C_p[G]}(1-\gamma u\mid {\rm Div}^0(\overline S)\otimes\mathbb C_p),\\
&\\
Q(u)&: =\frac{\prod_{v\in \Sigma}(1-\sigma_v^{-1}\cdot(qu)^{d_v})}{\det_{\mathbb C_p[G]}(1-\gamma\cdot qu\mid \mathbb C_p[G/H])}\cdot{\rm det}_{\mathbb C_p[G]}(1-\gamma\cdot u\mid \left(H_{1, \rm cris}\otimes_W\mathbb C_p\right)_{>0})\,.
\end{array}$$
\medskip

\begin{proposition}\label{poly-integrality} With notations as above, we have the following.
\begin{enumerate} \item
$P(u)={\rm det}_{\mathbb Z_p[G]}(1-\gamma u\mid T_p(M_{\overline S, \overline \Sigma}))\,.$
\item $P(u)\in\mathbb Z_p[G][u]\,.$
\item $Q(u)\in\mathbb Z_p[G][u]\,.$
\end{enumerate}
\end{proposition}
\begin{proof} If we combine (\ref{etale-jacobian}) above  with the definition of $P$, we obtain the following equalities.
$$\begin{array}{ll}
   P(u)& ={\rm det}_{\mathbb C_p[G]}(1-\gamma\cdot u\mid H_{1,p}\otimes_{\mathbb Z_p}\mathbb C_p)\cdot{\rm det}_{\mathbb C_p[G]}(1-\gamma u\mid {\rm Div}^0(\overline S)\otimes\mathbb C_p) \\
       & ={\rm det}_{\mathbb C_p[G]}(1-\gamma\cdot u\mid T_p(J_{\overline X})\otimes_{\mathbb Z_p}\mathbb C_p)\cdot{\rm det}_{\mathbb C_p[G]}(1-\gamma u\mid {\rm Div}^0(\overline S)\otimes\mathbb C_p)\\
       &={\rm det}_{\mathbb C_p[G]}(1-\gamma\cdot u\mid T_p(M_{\overline S, \emptyset})\otimes_{\mathbb Z_p}\mathbb C_p)\,.
  \end{array}$$
  However, Remark \ref{p-torsion} shows that we have an equality of of $\mathbb Z_p[G]$--modules
  $$T_p(M_{\overline S, \overline\Sigma}) = T_p(M_{\overline S, \emptyset})\,,$$
  and Theorem \ref{main-ct} shows that these modules are $\mathbb Z_p[G]$--projective of finite rank. Therefore,
  we have
  $${\rm det}_{\mathbb C_p[G]}(1-\gamma\cdot u\mid T_p(M_{\overline S, \emptyset})\otimes_{\mathbb Z_p}\mathbb C_p)={\rm det}_{\mathbb Z_p[G]}(1-\gamma u\mid T_p(M_{\overline S, \overline \Sigma}))\,,$$
  which proves part (1). (Note that we have applied equalities (\ref{base-change}) to base-change from
$\mathbb C_p[G]$ down to $\mathbb Z_p[G]$.)
  Of course, part (2) is a consequence of part (1).
  In order to prove part (3), recall the following elementary result.

  \begin{lemma} Let $R$ be a subring of the commutative ring $R'$ with $1$. Let $P, Q\in S[u]$ be two polynomials, such that
  $P(0)=1$, $P\in R[u]$ and $P\cdot Q\in R[u]$. Then $Q\in R[u]$.
  \end{lemma}
  \begin{proof} See Lemma A.5 in the Appendix of \cite{Popescu-Stark}. \end{proof}

  \noindent Now, apply the Lemma above to our polynomials $P$ and $Q$ and rings $R:=\mathbb Z_p[G]$ and $R':=\mathbb C_p[G]$ and keep in mind
  that $P\in R[u]$ (which is (2) above), $P(0)=1$ (from the definition of $P$ as characteristic polynomial) and that
  $P\cdot Q=\Theta_{S, \Sigma}(u)\in\mathbb Z[G][u]\subseteq \mathbb Z_p[G][u]\,.$ This concludes the proof of (3).
  \end{proof}
  Our next goal is to prove the following.
  \begin{proposition}\label{Q-invertible}
  We have $Q(\gamma^{-1})\in\mathbb Z_p[[\overline G]]^\times\,.$
  \end{proposition}
\begin{proof} We need the following elementary Lemma.

\begin{lemma} Let $\mathcal G$ be a profinite abelian group and let $x$ be an element in the profinite
group ring $\mathbb Z_p[[\mathcal G]]$. The following are equivalent.
\begin{enumerate}\item $x\in \mathbb Z_p[[\mathcal G]]^\times\,.$
\item For every continuous $p$--adic character $\psi$ of $\mathcal G$, we have $\psi(x)\in O_{\psi}^\times$,
where $O_\psi$ denotes the finite extension of $\mathbb Z_p$ generated by the values of $\psi$.
\end{enumerate}
\end{lemma}
\begin{proof} We have $\mathbb Z_p[[\mathcal G]]=\underset\longleftarrow\lim\, \mathbb Z_p[[\mathcal G/\mathcal H]]$ and therefore $\mathbb Z_p[[\mathcal G]]^\times= \underset\longleftarrow\lim\, \mathbb Z_p[[\mathcal G/\mathcal H]]^\times$,
where the projective limits are taken with respect to all the subgroups $\mathcal H$ of finite index in $\mathcal G$. Since every continuous character of $\mathcal G$ factors through a character of
$\mathcal G/\mathcal H$, for some $\mathcal H$ as above, it suffices to show the equivalence of (1) and (2) in the situation where $\mathcal G$ is finite. However, in the case where $\mathcal G$ is finite,
we have an equality $\mathbb Z_p[[\mathcal G]]=\mathbb Z_p[\mathcal G]$ and an injective ring morphism
$$\mathbb Z_p[\mathcal G]\longrightarrow \oplus_\psi O_\psi\,,\quad y\longrightarrow \oplus_\psi \psi(y)\,,$$
which induces an integral extension of rings, where the sum is taken
with respect to all the irreducible $\mathbb C_p$--valued characters
$\psi$ of $\mathcal G$. Since it is well known (and quite
elementary) that in an integral ring extension $\mathcal
A\subseteq\mathcal  B$ we have $x\in\mathcal A^\times$ if and only
if $x\in\mathcal B^\times$, for all $x\in \mathcal A$, the
equivalence of (1) and (2) is established.
\end{proof}
\noindent Now, we are ready to prove Proposition \ref{Q-invertible}. First, let us note that any continuous $p$--adic character $\psi$ of $\overline G=G\times\Gamma$
is a product $\psi=(\chi, \rho)$, where $\chi$ is a $p$--adic character of $G$ and $\rho$ is a continuous $p$--adic character of $\Gamma$. Therefore, we have
$$\psi(Q(\gamma^{-1}))=Q^\chi(\rho(\gamma^{-1})),$$
where $Q^\chi(u)$ is the polynomial in $O_\chi[u]$ obtained by evaluating $\chi$ at the coefficients of $Q(u)\in\mathbb Z_p[G][u]$. In order to simplify notation,
let us write $Q(u)=Q_{\Sigma}(u)\cdot Q_{\rm cris}(u)$, where
$$Q_{\Sigma}(u):=\frac{\prod_{v\in \Sigma}(1-\sigma_v^{-1}\cdot(qu)^{d_v})}{\det_{\mathbb C_p[G]}(1-\gamma\cdot qu\mid \mathbb C_p[G/H])}, \quad Q_{\rm cris}(u):={\rm det}_{\mathbb C_p[G]}(1-\gamma\cdot u\mid \left(H_{1, \rm cris}\otimes_W\mathbb C_p\right)_{>0}).$$
Note that in general we have $Q_\Sigma(u), Q_{\rm cris}(u)\in \mathbb C_p[G][u]\setminus\mathbb Z_p[G][u]$, but $Q(u)\in\mathbb Z_p[G][u]$ (see Proposition \ref{poly-integrality}.) Now, let $\psi=(\chi, \rho)$ be a character as above.
Remark \ref{sigma} implies that we have

$$
\begin{array}{ll}
Q_\Sigma^\chi(u)=&\left\{
                     \begin{array}{ll}
                       \prod\limits_{v\in \Sigma}(1-\chi(\sigma_v^{-1})\cdot(qu)^{d_v}), & \hbox{$\text{ if } \chi\mid_H\ne 1_H;$} \\
                       (\sum\limits_{i=0}^{d_{v_0}-1}\chi^i(\overline\gamma^{-1}) q^i\cdot u^i)\,\cdot\prod\limits_{v\in \Sigma\setminus\{v_0\}}(1-\chi(\sigma_v^{-1})\cdot(qu)^{d_v}), & \hbox{$\text{ if } \chi\mid_H= 1_H.$}
                     \end{array}
                   \right.\\
&\\
Q_{\rm cris}^\chi(u)=&\prod\limits_{i=1}^{d_\chi}(1-\alpha_{i, \chi}\cdot u)\,,\end{array}
$$
where $v_0\in\Sigma$ is arbitrary, and the $\alpha_{i,\chi}$'s are the eigenvalues of $\gamma$ acting on the $d_\chi$--dimensional $\mathbb C_p$--vector space $\left(H_{1, \rm cris}\otimes_W\mathbb C_p\right)_{>0}^\chi$, which is
the $\chi$--eigenspace of $\left(H_{1, \rm cris}\otimes_W\mathbb C_p\right)_{>0}$ with respect to the action of $G$. Now, let $O$ denote a large enough finite integral extension of $W$ inside $\mathbb C_p$, containing the
values of $\chi$ and $\rho$ and the eigenvalues $\alpha_{i,\chi}$  for all $i=1,\dots,d_\chi$. Note that since $\gamma$ acts $W$--linearly on the free, finite rank $W$--module $H_{1, \rm cris}$, the eigenvalues $\alpha_{i,\chi}$ are integral over $W$, for all $\chi$ and $i=1,\dots,d_\chi$. Further, if we let $\mathfrak m_O$ denote the maximal ideal of $O$, by the definition of $\left(H_{1, \rm cris}\otimes_W\mathbb C_p\right)_{>0}$, we have $\alpha_{i,\chi}\in\mathfrak m_O$, for all $\chi$ and $i$ as above. This observation combined with the obvious fact that $q\in\mathfrak m_O$ implies that
$$Q_\Sigma^\chi(u),\,Q_{\rm cris}^\chi(u),\, Q^\chi(u)\in 1+u\cdot\mathfrak m_O[u]\,.$$
Consequently, we have
$$\psi(Q(\gamma^{-1}))=Q^\chi(\rho(\gamma^{-1}))\in 1+\mathfrak m_O\subseteq O^\times\,.$$
However, since $Q(\gamma^{-1})\in\mathbb Z_p[[\overline G]]$, we know that $\psi(Q(\gamma^{-1}))\in O_{\psi}$. Consequently,
$$\psi(Q(\gamma^{-1}))\in O_\psi\cap O^\times = O_\psi^\times\,.$$
According to the Lemma above, this concludes the proof of Proposition \ref{Q-invertible}.
\end{proof}
Now, we are ready to state and prove the main result of this subsection.

\begin{theorem}\label{Fitting-p} Under the above hypotheses, the following hold.
\begin{enumerate}
\item The $\mathbb Z_p[G]$--module $T_p(M_{\overline S, \overline\Sigma})$ is projective.
\item
We have the following equality of ideals in $\mathbb Z_p[[\overline G]]$.
$$\left(\Theta_{S, \Sigma}(\gamma^{-1})\right)=\text{\rm Fit}_{\mathbb Z_p[[\overline G]]}(T_p(M_{\overline S, \overline\Sigma}))\,.$$
\end{enumerate}
\end{theorem}
\begin{proof} As stated in the proof of Proposition \ref{poly-integrality}, part (1) above is a particular case of Theorem \ref{main-ct}
(the case where $\kappa_0=\mathbb F_q$.) Now, Propositions \ref{poly-integrality}(1) and \ref{Q-invertible} imply that we have the following equalities of ideals in $\mathbb Z_p[[\overline G]]$.
\begin{equation}\label{ideals}
\left(\Theta_{S, \Sigma}(\gamma^{-1})\right)=(P(\gamma^{-1}))=({\rm det}_{\mathbb Z_p[G]}(1-\gamma \cdot u\mid T_p(M_{\overline S, \overline \Sigma}))\mid_{u=\gamma^{-1}})\,.
\end{equation}
Now, part (2) of our Theorem \ref{Fitting-p} is a direct consequence of equalities (\ref{ideals}) and Proposition \ref{fitting-calculation} applied to $M:=T_p(M_{\overline S, \overline\Sigma})$,
$R:=\mathbb Z_p[G]$, $\mathcal G:=\Gamma$, and $g:=\gamma$. Note that $R$ is indeed compact in the $p$--adic topology and semi-local (see the argument which ends the proof of Theorem \ref{Fitting-ell}.)
\end{proof}
\noindent Now, we combine Theorems \ref{Fitting-ell} and
\ref{Fitting-p} to obtain Theorem \ref{Fitting}, which completes the
task set at the beginning of this section.
\medskip

Next, we derive a corollary, which will be useful in the next
section. For that purpose, let $K_{\infty}:=K\mathbb F$ (field
compositum viewed inside some separable closure of $K$.) This is the
field of rational functions of the smooth, projective, irreducible
curve $X_\infty:=X\times_{\mathbb F_r}\mathbb F$, where $\mathbb F_r:=\mathbb
F\cap K$ is the exact field of constants of $K$. Let $G_\infty:={\rm
Gal}(K_\infty/k)$. Since $\Gamma$ is a free abelian pro-finite
group, we have a (non-canonical) isomorphism $G_\infty\simeq
H\times\Gamma$, where $H:={\rm Gal}(K_\infty/k_\infty)\simeq {\rm
Gal}(K/K\cap k_\infty)$. On the other hand, Galois theory and the
natural isomorphism $\Gamma:={\rm Gal}(\mathbb F/\mathbb F_q)\simeq{\rm
Gal}(k_\infty/k)$ permit us to identify $G_\infty$ with the subgroup
of $\overline G:=G\times\Gamma$, consisting of all $(g, \sigma)$,
with $g\in G$ and $\sigma\in\Gamma$, such that
$\pi_G(g)=\pi_{\Gamma}(\sigma)$, where $\pi_G:{\rm
Gal}(K/k)=G\twoheadrightarrow G/H={\rm Gal}(k_\infty\cap K/k)\text{
and }\pi_\Gamma:\Gamma={\rm Gal}(k_\infty/k)\twoheadrightarrow {\rm
Gal}(k_\infty\cap K/k)=G/H$ are the usual projections induced by
Galois restriction. There is an exact sequence in the category of
groups
$$\xymatrix{
1\ar[r] &G_{\infty}\ar[r] &\overline G\ar[r] &G/H\ar[r] &1 },$$
where
the injection sends $\tau\in G_{\infty}$ to $(\tau\mid_{K},
\tau\mid_{{k}_{\infty}})$ and the surjection sends $(g, \sigma)$ in
$G\times\Gamma$ to $\pi_G(g)\pi_{\Gamma}(\sigma)^{-1}$ in $G/H$.
This leads to a canonical identification of $\mathbb
Z_{\ell}[[G_\infty]]$ with a subring of $\mathbb Z_{\ell}[[\overline
G]]$. Since for every prime $v$ in $k$, which is unramified in
$K/k$, we have $\pi_G(\sigma_v)=\pi_{\Gamma}(\gamma^{d_v})$, the
product formula (\ref{product-formula}) shows that, under the above
identification, we have
$$\Theta_{S, \Sigma}(\gamma^{-1})\in\mathbb Z_{\ell}[[G_\infty]]\subseteq\mathbb Z_{\ell}[[\overline G]]\,.$$
Now, we let $M_{S_{\infty}, \Sigma_{\infty}}$ be the Picard $1$--motive associated to $(X_{\infty}, \mathbb F, S_\infty, \Sigma_{\infty})$, where
$S_\infty$ and $\Sigma_{\infty}$ are the sets of closed points on $X_{\infty}$  sitting above points in $S$ and $\Sigma$, respectively.
Its $\ell$--adic realizations $T_{\ell}(M_{S_{\infty}, \Sigma_{\infty}})$ are endowed with natural $\mathbb Z_{\ell}[[G_{\infty}]]$--module structures, for all primes $\ell$.

\begin{corollary}\label{Fitting-connected} For every prime number $\ell$, the following hold.
\begin{enumerate} \item $T_{\ell}(M_{S_{\infty}, \Sigma_{\infty}})$ is a projective $\mathbb Z_{\ell}[H]$--module.
\item We have an equality of $\mathbb Z_{\ell}[[G_\infty]]$--ideals
${\rm Fit}_{\mathbb Z_{\ell}[[G_\infty]]}(T_{\ell}(M_{S_{\infty}, \Sigma_{\infty}}))=(\Theta_{S, \Sigma}(\gamma^{-1}))\,.$
\end{enumerate}
\end{corollary}
\begin{proof}
The proof of Theorem \ref{main-ct} gives an isomorphism of $\mathbb Z_{\ell}[G]$--modules
$$T_{\ell}(M_{\overline S, \overline \Sigma})\simeq T_{\ell}(M_{S_{\infty}, \Sigma_{\infty}})\otimes_{\mathbb Z_{\ell}[H]}\mathbb Z_{\ell}[G]\,,$$
for all primes $\ell$. This isomorphism combined with Theorem \ref{Fitting}(1) implies part (1) of the Corollary.
In order to prove part (2), let $R_\infty:=\mathbb Z_{\ell}[[G_\infty]]$ and $\overline R:=R[[\overline G]]$. Since $\overline G/G_{\infty}\simeq G/H$ (see above),
$\overline R$ is a free $R_{\infty}$--module of rank $|G/H|$. Consequently, $\overline R$ is a faithfully flat $R_{\infty}$--algebra. The isomorphism above can be
re-written as an isomorphism of $\overline R$--modules
$$T_{\ell}(M_{\overline S, \overline \Sigma})\simeq T_{\ell}(M_{S_{\infty}, \Sigma_{\infty}})\otimes_{R_{\infty}}\overline R\,.$$
Consequently, since Fitting ideals commute with extension of scalars, the isomorphism above combined with Theorem \ref{Fitting}(2) gives an equality of $\overline R$--ideals
$${\rm Fit}_{R_{\infty}}(T_{\ell}(M_{S_{\infty}, \Sigma_{\infty}}))\overline R = {\rm Fit}_{\overline R}(T_{\ell}(M_{\overline S, \overline\Sigma}))=\Theta_{S, \Sigma}(\gamma^{-1})\cdot\overline R$$
Now, recall that since $\overline R$ is a faithfully flat $R_{\infty}$--algebra, we have $I\overline R\cap R_{\infty}=I$, for all ideals $I$ in $R_{\infty}$ (see \cite{Matsumura}, Theorem 7.5(ii), p.49.)
If we apply this property to the $R_{\infty}$--ideals ${\rm Fit}_{R_{\infty}}(T_{\ell}(M_{S_{\infty}, \Sigma_{\infty}}))$ and $(\Theta_{S,\Sigma}(\gamma ^{-1}))$ and take into account the equality of
$\overline R$--ideals above, we obtain
$${\rm Fit}_{R_{\infty}}(T_{\ell}(M_{S_{\infty}, \Sigma_{\infty}}))=(\Theta_{S, \Sigma}(\gamma^{-1}))={\rm Fit}_{\overline R}(T_{\ell}(M_{\overline S, \overline\Sigma}))\cap R_{\infty},$$
which concludes the proof of part (2).
\end{proof}

\section{\bf Refinements of the Brumer-Stark and Coates-Sinnott conjectures for global fields of characteristic $p$.}
\noindent In this section, we prove that Theorem \ref{Fitting} (or, more precisely, Corollary \ref{Fitting-connected}) above implies refinements of the Brumer-Stark and Coates-Sinnott
conjectures, linking special values of equivariant Artin $L$--functions to certain Galois module structure invariants of
ideal--class groups and $\ell$--adic \'etale cohomology groups, in the context of abelian extensions of characteristic $p$ global fields. We are working with the notations and under the hypotheses
of the previous section.

\subsection{Generalized Jacobians, class--field theory and $\ell$--adic \'etale cohomology}
 If $w$ is a closed point on $X$, we denote by $\mathbb F_r(w)$ its residue field and let $\text{deg}(w):=[\mathbb F_r(w):\mathbb F_r]$ denote its degree over $\mathbb F_r$.
As customary, the degrees of divisors on $X$ are computed over the strict field of constants $\mathbb F_r$ of $K$. For every $n\in\mathbb Z_{\geq 1}$, we let $K_n$ denote the field compositum $K_n:=K\mathbb\cdot F_{r^n}$,
which is a characteristic $p$ global field of exact field of constants $\mathbb F_{r^n}$. We let $\Gamma_K=\text{Gal}(K_\infty/K)$. Obviously, we have an isomorphism
of topological groups $\Gamma_K\simeq\widehat{\mathbb Z}$. We denote by $\gamma_K$ the canonical topological generator of $\Gamma_K$, given by the $r$--power arithmetic Frobenius. We let $X_n:=X\times_{\mathbb F_r}\mathbb F_{r^n}$
denote the smooth, projective models of $K_n$ over $\mathbb F_{r^n}$, for all $n$.

\begin{definition} Let $T$ be a finite (possibly empty) set of closed points in $X$. For $\ast\in\{\infty\}\cup\mathbb N$, we let $T_\ast$ be the set of closed points on $X_\ast$ sitting above points in $T$.
We define the $T$--modified Picard groups associated to $K_\ast$
(or to $X_\ast$, for that matter)
$${\rm Pic}_{T}(K_\ast):=\frac{{\rm Div}(X_\ast\setminus T_\ast)}{\{{\rm div}(f)\,\mid\,f\in K_{\ast,{T}}^\times\}}\,,\qquad {\rm Pic}^0_{T}(K_\ast):=\frac{{\rm Div}^0(X_\ast\setminus T_\ast)}{\{{\rm div}(f)\,\mid\,f\in K_{\ast, {T}}^\times\}}\,,$$
where  $K_{\ast, {T}}^\times:=\{f\in K_\ast^\times\,\mid\, f\equiv 1\mod w, \text{\rm  for all } w\in T_\ast\}$, as in Definition \ref{st}. (Obviously, $K_1=K$ and $T_1=T$.)
\end{definition}
\noindent For $T=\emptyset$, one obtains the classical Picard groups and $\emptyset$ will be dropped from the notation. We have an obvious commutative diagram with exact rows and columns
\begin{equation}\label{Pic-diagram} \xymatrix{
& & 0\ar[d] & 0\ar[d] &\\
0\ar[r] & \dfrac{\oplus_{w\in T }\mathbb F_r(w)^\times}{\mathbb F_r^\times}\ar[r]\ar[d]^{=} & \text{Pic}^0_T(K)\ar[r]\ar[d] &\text{Pic}^0(K)\ar[r]\ar[d] &0\\
0\ar[r] & \dfrac{\oplus_{w\in T}\mathbb F_r(w)^\times}{\mathbb F_r^\times}\ar[r] & \text{Pic}_T(K)\ar[r]\ar[d]^{\text{deg}} &\text{Pic}(K)\ar[r]\ar[d]^{\text{deg}} &0\\
 & & \mathbb Z\ar[r]^{=}\ar[d] &\mathbb Z \ar[d] &\\
  & & 0 & 0 &}
\end{equation}
at the $K=K_1$ level and exact analogues at the $K_\ast$--levels, for all $\ast$ as above.
The surjectivity of the degree maps (``$\text{deg}$'') at the finite levels is a classical theorem of F. K. Schmidt. If $T\ne\emptyset$, we view $\mathbb F_r^\times$ as sitting inside $\oplus_{w\in T}\mathbb F_r(w)^\times$ diagonally.
The injective maps in the horizontal short exact sequences send the
class modulo $\mathbb F_r^\times$ of an element $(x_w)_w\in \oplus_{w\in T}\mathbb F_r(w)^\times$ to the class of $\text{div}(f)$ of a function $f\in K^\times$
satisfying $f\equiv x_w\mod w$, for all $w\in T$. The existence of $f$ is implied by the weak approximation theorem.

\begin{remark}\label{generalized-jacobian} Note that, for any $T$ as above, we have
${\rm Pic}^0_{T}(K_\infty)=J_{T_\infty}(\mathbb F)\,,$
where $J_{T_\infty}$ is the semi-abelian variety (generalized Jacobian)
associated to the set of data $(X_\infty, T_\infty)$ over $\mathbb F$
(see \S2.)
\end{remark}

\begin{remark}\label{injective} For all $T$ as above and all natural numbers $m, n$ with $n\mid m$,  the canonical maps
$${\rm Pic}^0_T(K_n)\longrightarrow {\rm Pic}^0_T(K_m)\text{  and  } {\rm Pic}^0_T(K_n)\longrightarrow {\rm Pic}^0_T(K_\infty)$$
are injective. Indeed, this follows immediately from the equalities $H^1(\Gamma_{m,n}, \mathbb F_{r^m}^\times)=H^1(\Gamma_n, \mathbb F^\times)=1$, where
$\Gamma_n:={\rm Gal}(K_\infty/K_n)$ and $\Gamma_{m, n}:={\rm Gal}(K_m/K_n)$.
If we identify ${\rm Pic}^0_T(K_n)$ with a subgroup of ${\rm Pic}^0_T(K_\infty)$ under the above injective maps,
then we have equalities
\begin{equation}\label{invariants}
J_{T_\infty}(\mathbb F)^{\Gamma_n}={\rm Pic}^0_T(K_\infty)^{\Gamma_n}={\rm Pic}^0_T(K_n)\,,\qquad J_{T_\infty}(\mathbb F)=\bigcup\limits_n {\rm Pic}^0_T(K_n).
\end{equation}
Indeed, this is an immediate consequence of the fact that $K_\infty/K_n$ is everywhere unramified and
$H^1(\Gamma_n, K^\times_{\infty, T})=1$, for all $n$. (See Step 2 in the proof of Theorem \ref{G-invariants}.)
\end{remark}

\begin{remark}\label{class-field} Global class-field theory establishes a canonical injective (Artin reciprocity) morphism
$$\rho_{n, T}: {\rm Pic}_T(K_n)\longrightarrow {\mathfrak X}_{n, T}\,,$$
where ${\mathfrak X}_{n, T}$ is the Galois group of the maximal abelian extension $M_{n, T}$
of $K_n$ which is unramified outside of $T_n$ and at most tamely ramified at $T_n$. Moreover, the image of the above
morphism is dense in ${\mathfrak X}_{n, T}$ and it consists of all $\sigma\in {\mathfrak X}_{n, T}$ which via the Galois restriction map
${\mathfrak X}_{n, T}\longrightarrow {\rm Gal}(K_\infty/K_n)$ land in the subgroup $(\gamma_K^n)^{\mathbb Z}$ of ${\rm Gal}(K_\infty/K_n)=(\gamma_K^n)^{\widehat{\mathbb Z}}$.
Consequently, the morphism above induces an isomorphism at the level of profinite completions
$$\widehat{\rho_{n, T}}: \widehat{{\rm Pic}_T(K_n)}\overset\sim\longrightarrow {{\mathfrak X}_{n, T}}\,.$$
Note that there is a (non-canonical) group isomorphism $\widehat{{\rm Pic}_T(K_n)}\overset\sim\longrightarrow{{\rm Pic}^0_T(K_n)}\times\widehat{\mathbb Z}\,,$ coming
from a (non-canonical) splitting of the left-most vertical short exact sequence in the diagram (\ref{Pic-diagram}) above.
\end{remark}

\begin{remark}\label{class-field-infinity} Let ${\mathfrak X}_{\infty, T}$ denote the Galois group of the maximal abelian extension $M_{\infty, T}$ of $K_\infty$
which is unramified away from $T_\infty$ and at most tamely ramified at $T_{\infty}$. Since $K_\infty/K$ is unramified everywhere, it is easy to show that
$M_{\infty, T}=\cup_nM_{n, T}$ (union viewed inside a fixed separable closure of $K_\infty$.) Consequently, we have an isomorphism of topological groups
$$\underset n{\underset{\longleftarrow}\lim}{\mathfrak X}_{n, T}\overset\sim\longrightarrow {\mathfrak X}_{\infty, T}\,,$$
where the projective limit is taken with respect to the Galois restriction maps ${\rm res}_{m, n}:{\mathfrak X}_{m, T}\to {\mathfrak X}_{n, T}$,
for all $n, m$, such that $n\mid m$. Now, elementary properties of the Artin reciprocity map leads to a canonical isomorphism of topological groups
$$\underset n{\underset{\longleftarrow}\lim}\, \widehat{{\rm Pic}_T(K_n)}\overset\sim\longrightarrow {\mathfrak X}_{\infty, T}\,,$$
where the projective limit is taken with respect to the norm maps $N_{m, n}: {\rm Pic}_T(K_m)\longrightarrow {\rm Pic}_T(K_n)$, for all $m$, $n$ with $n\mid m$.
However, as the reader can easily check, there are commutative diagrams
$$\xymatrix
{0\ar[r] &{\rm Pic}^0_{T}(K_m)\ar[r]\ar[d]^{N^0_{m, n}} &\widehat{ {\rm Pic}_{T}(K_m)}\ar[r]^{\quad \widehat{\rm deg}} \ar[d]^{\widehat{N_{m, n}}} &\widehat{\mathbb Z}\ar[r]\ar[d]^{\times\frac{m}{n}} &0\\
0\ar[r] &{\rm Pic}^0_{T}(K_n)\ar[r] & \widehat{{\rm Pic}_{T}(K_n)}\ar[r]^{\quad \widehat{\rm deg}} &\widehat{\mathbb Z}\ar[r] &0}$$
for all $m$ and $n$ as above, where $N_{m, n}^0$ denotes the restriction of $N_{m, n}$ to ${\rm Pic}^0_{T}(K_m)$.  Since $\underset n{\underset{\longleftarrow}\lim}\,\,\widehat{\mathbb Z}=0$, we obtain canonical isomorphisms of topological groups
\begin{equation}\label{isomorphism-Galois}
\underset n{\underset{\longleftarrow}\lim}\,{\rm Pic}^0_T(K_n)\overset\sim\longrightarrow \underset n{\underset{\longleftarrow}\lim}\, \widehat{{\rm Pic}_T(K_n)}\overset\sim\longrightarrow {\mathfrak X}_{\infty, T}\,.\end{equation}
\end{remark}

\begin{lemma}
Let $M$ be a torsion, divisible $\mathbb Z_{\ell}$--module of finite co-rank, endowed with a
$\Gamma_K$--action which is continuous with respect to the discrete and profinite topology on $M$ and $\Gamma_K$, respectively.
Assume that $M^{\Gamma_n}$ is finite, for all $n$.  Then, the following hold.
\begin{enumerate}
\item There exist canonical isomorphisms $\phi_n: M^{\Gamma_n}\overset\sim\longrightarrow T_{\ell}(M)_{\Gamma_n}$ and commutative diagrams
$$\xymatrix{
M^{\Gamma_m}\ar[d]^{N_{m, n}}\ar[r]_{\sim\quad }^{\phi_m\quad } & T_{\ell}(M)_{\Gamma_m}\ar[d]\\
M^{\Gamma_n}\ar[r]^{\phi_n\quad }_{\sim\quad } & T_{\ell}(M)_{\Gamma_n}\,,
}$$
for all $n$ and $m$ with $n\mid m$, where $N_{m,n}$ is the usual norm (multiplication by $\frac{\gamma_K^m-1}{\gamma_K^n-1}$) map and the vertical maps on the right are
the natural projections.
\item There are natural isomorphisms of topological compact $\mathbb Z_{\ell}[[\Gamma_K]]$--modules
$$\underset n{\underset\longleftarrow\lim}\, M^{\Gamma_n}\overset\sim\longrightarrow \underset n{\underset\longleftarrow\lim}\, T_{\ell}(M)_{\Gamma_n}\overset\sim\longrightarrow T_\ell(M).$$
The first isomorphism above is the projective limit of the $\phi_n$'s and the second is the inverse of the projective limit of the canonical surjections $\pi_n: T_{\ell}(M)\twoheadrightarrow T_{\ell}(M)_{\Gamma_n}$.
\end{enumerate}
\end{lemma}
\begin{proof} The proof is routine, except for two subtle points: the definition of the maps $\phi_n$ and the bijectivity of $\underset n{\underset\longleftarrow\lim}\, \pi_n$. We will explain these two
points, leaving the details to the reader.

$\bullet$ {\bf Constructing $\phi_n$.} Let us fix $n\in\mathbb N$ and $x\in M^{\Gamma_n}$. Since $M$ is divisible, there exists
$$(x_r)_{r\geq 1}\in \underset r{\underset\longleftarrow\lim} M,\text{  such that }x_1=x\,,$$
where the projective limit is taken with respect to the
multiplication by $\ell$ maps
$M\overset{\times\ell}\longrightarrow M$. Since $x\in
M^{\Gamma_n}$, if we set $y_r:=x_{r+1}^{\gamma_K^n-1}$, for all
$r\geq 1$,  we have $(y_r)_{r\geq 1}\in T_{\ell}(M)$. By
definition, the map $\phi_n$ sends $x$ to the class of
$(y_r)_{r\geq 1}$ in $T_\ell(M)_{\Gamma_n}$. One can easily check
that this definition is independent of any choices, $\phi_n$ is an
isomorphism and the diagram in (1) above is indeed commutative. In
fact, the maps $\phi_n$ arise naturally as connecting morphisms in
a snake lemma six-term exact sequence.
\medskip

$\bullet$ {\bf Proving the bijectivity of $\underset n{\underset\longleftarrow\lim}\, \pi_n$.} The bijectivity is obviously equivalent to the equality
$$\bigcap_n (1-\gamma_K^n)T_{\ell}(M)=0\,.$$
Since $M=\cup_n M^{\Gamma_n}$ ($\Gamma_K$ acts continuously on $M$ !) and $M$ has finite co-rank, for all $m$, there exists an $n$, such that
$M[\ell^m]\subseteq M^{\Gamma_n}$. This fact combined with the isomorphisms $\phi_n$ implies that
$$\bigcap_n (1-\gamma_K^n)T_{\ell}(M)\subseteq \bigcap_m \ell^m T_{\ell}(M)=0\,.$$
\end{proof}

\begin{corollary}\label{galois-semiabelian} For all primes $\ell$, we have canonical isomorphisms of $\mathbb Z_{\ell}[[\Gamma_K]]$--modules
$${\rm Pic}^0_{T}(K)^{(\ell)}\overset\sim\longrightarrow T_{\ell}(J_{T_\infty})_{\Gamma_K},\qquad  T_{\ell}(J_{T_\infty})\overset\sim\longrightarrow \mathfrak X_{\infty, T}^{(\ell)},$$
where ${\rm Pic}^0_{T}(K)^{(\ell)}:={\rm Pic}^0_{T}(K)\otimes\mathbb Z_{\ell}$ and $\mathfrak X_{\infty, T}^{(\ell)}:=\mathfrak X_{\infty, T}\otimes_{\mathbb Z}\mathbb Z_\ell$. ({\bf Note} that $\mathfrak X_{\infty, T}^{(\ell)}$ is the Galois group of the maximal abelian pro-$\ell$
extension of $K_\infty$ which is unramified away from $T_\infty$ if $\ell\ne p$ and unramified away from $T_\infty$ and at most tamely ramified at $T_\infty$, if $\ell=p$.)
\end{corollary}
\begin{proof} Let $M:=J_{T_\infty}[\ell^\infty]$. This is a torsion, divisible, $\mathbb Z_{\ell}$--module of finite co-rank. (Theorem \ref{ml-free} with $G$ trivial, $Z'=X_\infty$, $\mathcal S'=\emptyset$ and $\mathcal T'=T_\infty$
gives an exact formula for the co-rank.)
The continuity of the $\Gamma_K$--action on $M$ and the finiteness of $M^{\Gamma_n}$
are direct consequences of equalities (\ref{invariants}) and the finiteness of ${\rm Pic}^0_{T}(K_n)$, for all $n$. Now, apply part (2) of the preceding Lemma, combined
with equalities (\ref{invariants}) and isomorphism (\ref{isomorphism-Galois}) to conclude the proof of the corollary.
\end{proof}
\medskip

Next, we will use Corollary \ref{galois-semiabelian}  to express certain
$\ell$--adic \'etale cohomology groups associated to $K$ in terms
of the $\ell$--adic realizations $T_\ell(J_{T_\infty})$. Let $T$ be as above.
Assume that $T\ne\emptyset$. Further, recalling that $K$ is the top field of a Galois extension $K/k$ of
Galois group $G$, let us assume that $T$ is invariant under the $G$--action on primes in $K$.
In this case, if $\ell$ is an arbitrary prime, the modules
$T_{\ell}(J_{T_\infty})$, $\mathfrak X_{\infty,
T}^{(\ell)}$ and ${\rm Pic}^0_{T}(K)^{(\ell)}$ have
obvious natural $\mathbb Z_{\ell}[[G_\infty]]$--module structures
and the isomorphisms in Corollary \ref{galois-semiabelian}
preserve those structures.

Now, let $\ell$ be a prime number with $\ell\ne p$ and let
$n\in\mathbb Z_{\geq 0}$. As usual, for all $i\in\mathbb Z_{\geq 0}$, we
denote by $H^i_{et}(O_{K, T}, \mathbb Z_{\ell}(n))$ the $i$--th
\'etale cohomology group with the coefficients in the $\ell$--adic
sheaf $\mathbb Z_{\ell}(n)$ for the scheme ${\rm Spec}(O_{K,
T})$ associated to the subring of $T$--integers $O_{K,
T}$ in $K$. For the definition and main properties of these
$\ell$--adic \'etale cohomology groups, the reader can consult the
excellent survey \cite{Kolster}. Functoriality in \'etale
cohomology leads to natural $\mathbb Z_{\ell}[G]$--module structures
on $H^i_{et}(O_{K, T}, \mathbb Z_{\ell}(n))$, for all $i$ and
$n$ as above. In what follows, we denote by
$$\kappa_\ell: G_\infty \longrightarrow{\rm Aut}(\mu_{\ell^\infty})\overset\sim\longrightarrow \mathbb Z_{\ell}^\times$$
the $\ell$--adic cyclotomic character over $k$, restricted to
$G_\infty$. It is the continuous character which factors through the
$\ell$--adic character of $\Gamma:=\Gamma_k={\rm Gal}(k_\infty/k)$
which sends $\gamma:=\gamma_k$ to $q$. Note that under the canonical
identification of $G_\infty$ with a subgroup of $\overline
G:=G\times\Gamma$ made in \S4, the character $\kappa_\ell$ can be
extended to the continuous character $\tilde{\kappa}_\ell$ of
$\overline G$ which is trivial on $G$ and sends $\gamma$ to $q$.

\begin{definition}\label{Tate-twists} Let $R$ be a commutative $\mathbb Z_{\ell}$--algebra, $M$ an  $R[[G_\infty]]$--module and $n\in\mathbb Z$.
\begin{enumerate} \item The Tate twist
$M(n)$ is the $R[[G_\infty]]$--module $M$ with the twisted
$G_\infty$--action
$$\sigma\ast m=\kappa_\ell(\sigma)^n\cdot{}^\sigma m,$$
for all $\sigma\in G_\infty$ and $ m\in M$ and the original
$R$--action.
\item We let $t_n: R[[G_\infty]]\simeq R[[G_\infty]]$
    be the $R$--algebra isomorphism which sends
    $\sigma\in G_\infty$ to
$$t_n(\sigma):=\kappa_\ell(\sigma)^n\cdot\sigma\,.$$
\end{enumerate}\end{definition}

\begin{remark}\label{fitting-twisting} With notations as above, if $M$
is a finitely generated $R[[G_\infty]]$--module, then
$${\rm Fit}_{R[[G_\infty]]}(M(n))=t_{-n}({\rm
Fit}_{R[[G_\infty]]}(M))\,,$$ for all $n\in \mathbb Z$. (See \cite{Popescu}, Lemma 3.1.)
\end{remark}

\begin{definition} Let $M$ be a $\mathbb Z_{\ell}[[\mathcal G]]$--module, where $\mathcal G$ is a profinite, abelian group.
We let $$M^\ast:={\rm Hom}_{\mathbb Z_\ell}(M, \mathbb
    Z_{\ell}), \qquad M^\vee:={\rm Hom}_{\mathbb Z_\ell}(M, \mathbb
    Q_\ell/\mathbb Z_{\ell}),$$ viewed as $\mathbb Z_{\ell}[[\mathcal
    G]]$--modules endowed with either the co-variant or the
    contra-variant $\mathcal G$--actions, depending on the
    context. The covariant and contra-variant actions are defined by ${}^gf(m)=f(g\cdot
    m)$ and ${}^gf(m)=f(g^{-1}\cdot m)$, respectively, for all
    $g\in\mathcal G$, $m\in M$ and $f\in M^{\ast}$ or
    $f\in M^{\vee}$.
\end{definition}

\begin{lemma}\label{etale-semiabelian} With notations as above, we have isomorphisms of $\mathbb Z_{\ell}[G]$--modules
$$H_{et}^2(O_{K, T}, \mathbb Z_{\ell}(n))\overset\sim\longrightarrow T_{\ell}(J_{T_\infty})(-n)_{\Gamma_K}^{\,\,\vee},$$
for all $n\in\mathbb Z_{\geq 2}$, where the dual to the right is endowed with the contra-variant $G$--action.
\end{lemma}
\begin{proof} The usual Hochschild-Serre spectral sequence argument in Galois cohomology, combined with the well-known finiteness of the groups
$H_{et}^i(O_{K, T}, \mathbb Z_{\ell}(n))$, for all $n\in\mathbb Z_{\geq 2}$ and all $i\in\mathbb Z_{\geq 0}$,  leads
to isomorphisms of $\mathbb Z_{\ell}[G]$--modules
$$H_{et}^2(O_{K, T}, \mathbb Z_{\ell}(n))\overset\sim\longrightarrow \mathfrak X_{\infty, T}^{(\ell)}(-n)_{\Gamma_K}^{\,\,\vee},$$
where the dual above is endowed with the contra-variant $G$--action (see the arguments in \cite{Kolster}, pp. 201--203 and pp. 237--238.)
The Galois group $\mathfrak X_{\infty, T}$ (and implicitly its $\ell$--primary component $\mathfrak X_{\infty, T}^{(\ell)}$) is endowed with the usual
lift-and-conjugation ${\rm Gal}(K_\infty/k)$--action. Now, we combine the isomorphism above with Corollary \ref{galois-semiabelian}
to arrive at the desired result. \end{proof}

\subsection{ The Brumer-Stark and Coates-Sinnott conjectures in characteristic $p$.} With notations as in \S4.2, we let $S_K$ and $\Sigma_K$
denote the sets consisting of all the primes in $K$ sitting above primes in $S$ and $\Sigma$, respectively.
However, for simplicity, the subscript $K$ will be dropped from the
notation whenever convenient (as in ${\rm Pic}^0_{\Sigma}(K)$ or $O_{K, S}$, for example.) Note that, by definition, both $S_K$ and $\Sigma_K$ are $G$--invariant.

\begin{conjecture}(Brumer-Stark) With notations as above, we have
$$\Theta_{S, \Sigma}(1)\in {\rm Ann}_{\mathbb Z[G]}({\rm Pic}^0_{\Sigma}(K))\,.$$
\end{conjecture}
\medskip

\begin{remark} {\bf (a)} For a fixed $S$ and all $\Sigma$, so that $(K/k, S, \Sigma)$ satisfy the required properties,
the statement above is equivalent to the characteristic $p$ version
of the classical Brumer-Stark conjecture for the given $S$ (as
stated in \cite{Tate-Stark}, Chpt.~V). In order to see this, one has
to make the following crucial observation. Let us fix $S$. For every
$\Sigma$, satisfying the above properties, let
$$\delta_\Sigma(u):=\prod\limits_{v\in \Sigma}(1-\sigma_v^{-1}\cdot(qu)^{d_v})\,.$$
Then the set $\{\delta_{\Sigma}(1)\mid \Sigma\}$ generates the ideal
${\rm Ann}_{\mathbb Z[G]}(\mu_K)$, where $\mu_K$ denotes the group of
roots of unity in $K$. (This is the exact analogue in characteristic
$p$ of Lemma 1.1 in \cite{Tate-Stark}, Chpt.~V.)

{\bf (b)} The conjecture above was proved independently and with different methods by Deligne (see \cite{Tate-Stark}, Chpt. V)
and Hayes (see \cite{Hayes}). In the next section, we will prove a refined version of this conjecture involving Fitting
ideals rather than annihilators of ideal class-groups.
\end{remark}
\medskip

\begin{conjecture} (Coates-Sinnott) With notations as above, we have
$$\Theta_{S, \Sigma}(q^{n-1})\in {\rm Ann}_{\mathbb Z_{\ell}[G]}(H_{et}^2(O_{K, S}, \mathbb Z_{\ell}(n)))\,,$$
for all primes $\ell\ne p$ and all $n\in\mathbb Z_{\geq 2}.$
\end{conjecture}

\begin{remark} For a fixed $S$ and all $\Sigma$, such that $(K/k, S, \Sigma)$ satisfy the above properties, the statement above is the exact characteristic $p$ analogue of
the classical Coates-Sinnott conjecture for that fixed $S$ (as stated in \cite{Coates-Sinnott}). In order to see this, one needs to make the following observation.
If we fix $S$ as above, then,
for all $n\in\mathbb Z_{\geq 2}$ and all primes $\ell\ne p$, the set $\{\delta_{\Sigma}(q^{n-1})\mid \Sigma\}$ generates
$${\rm Ann}_{\mathbb Z_{\ell}[G]}(H^1_{et}(O_{K, S}, \mathbb Z_{\ell}(n))_{\rm tors})$$ as a $\mathbb Z_{\ell}[G]$--ideal, where $\Sigma$ runs through all finite sets of primes in $k$, such that $S$ and $\Sigma$  satisfy the above properties.
Since $$H^1_{et}(O_{K, S}, \mathbb Z_{\ell}(n))_{\rm tors}=(\mathbb Q_{\ell}/\mathbb Z_{\ell}(n))^{\Gamma_K}$$ (see \cite{Kolster}, pp. 201-203), the statement above is the exact analogue in characteristic $p$
of Lemma 2.3 in \cite{Coates}. In order to compare the above observation to that in {\bf (a)} of the last remark, the reader is invited to note that $H^1_{et}(O_{K, S}, \mathbb Z_{\ell}(1))_{\rm tors}=
(\mathbb Q_{\ell}/\mathbb Z_{\ell}(1))^{\Gamma_K}=\mu_K^{(\ell)}$, for all $\ell\ne p$, while $\mu_K^{(p)}=1$.
\end{remark}

\subsection{ Refinements of the Brumer-Stark and Coates-Sinnott conjectures} Next, we prove the promised refinements of the two conjectures
stated in the previous section. The notations are the same as above.

\begin{lemma} Let $M$ be a $\mathbb Z_{\ell}[[G_\infty]]$--module, which is free of finite rank as a $\mathbb Z_{\ell}$--module and such that $M_{\Gamma_K}$ is
finite. Then, we have an isomorphism of $\mathbb Z_{\ell}[G]$--modules
$$(M^\ast)_{\Gamma_K}\overset\sim\longrightarrow (M_{\Gamma_K})^{\vee}\,,$$
where both dual modules are viewed with either the co-variant or
the contra-variant $G_\infty$--actions.
\end{lemma}
\begin{proof} We will prove the statement for the co-variant action and leave the remaining case to the reader.
Since $M$ is $\mathbb Z_{\ell}$--free of finite rank and $M_{\Gamma_K}$ is finite, we have $M^{\Gamma_K}=0$. Consequently, we have an exact sequence of $\mathbb Z_{\ell}[[G_\infty]]$--modules
$$\xymatrix{
0\ar[r] & M\ar[r]^{1-\gamma_K} &M\ar[r] & M_{\Gamma_K}\ar[r] &0.
}$$
Apply the functor ${\rm Hom}_{\mathbb Z_{\ell}}(\ast,\, \mathbb Z_{\ell})$ to the sequence above to obtain an exact sequence
$$\xymatrix{
0\ar[r] & M^\ast\ar[r]^{1-\gamma_K} &M^\ast\ar[r] & {\rm Ext}^1_{\mathbb Z_{\ell}}(M_{\Gamma_K}, \mathbb Z_{\ell})\ar[r] &0.
}$$
This shows that we have an isomorphism of $\mathbb Z_{\ell}[G]$--modules $(M^\ast)_{{\Gamma_K}}\simeq{\rm Ext}^1_{\mathbb Z_{\ell}}(M_{\Gamma_K}, \mathbb Z_{\ell})$. Now, one applies the functor
${\rm Hom}_{\mathbb Z_{\ell}}(M_{{\Gamma_K}},\, \ast)$ to the exact sequence
$$\xymatrix{0\ar[r] &\mathbb Z_{\ell}\ar[r] &\mathbb Q_{\ell}\ar[r] &\mathbb Q_{\ell}/\mathbb Z_{\ell}\ar[r] &0
}$$
to obtain an isomorphism of  $\mathbb Z_{\ell}[G]$--modules $(M_{{\Gamma_K}})^\vee\simeq {\rm Ext}^1_{\mathbb Z_{\ell}}(M_{\Gamma_K}, \mathbb Z_{\ell})$. Taking into account the
previous isomorphism, this concludes the proof.
\end{proof}
\begin{corollary}\label{iso-duals} The following hold.
\begin{enumerate}\item For all prime numbers $\ell$, we have an isomorphism of $\mathbb Z_{\ell}[G]$--modules
$$T_{\ell}(J_{\Sigma_\infty})^\ast_{\,\,\,{\Gamma_K}}\overset\sim\longrightarrow {{\rm Pic}^0_\Sigma(K)^{(\ell)}}^{\, \vee},$$
where the duals are $G_\infty$--co-variant.
\item For all prime numbers $\ell\ne p$ and all $n\in\mathbb Z_{\geq 2}$, we have an isomorphism of $\mathbb Z_{\ell}[G]$--modules
$$T_{\ell}(J_{S_\infty})(-n)^\ast_{\,\,\, {\Gamma_K}}\overset\sim\longrightarrow H^2_{et}(O_{K, S}, \mathbb Z_{\ell}(n)),$$
where the dual is $G_\infty$-contra-variant.
\end{enumerate}
\end{corollary}
\begin{proof} The idea is to apply the Lemma above to $M:=T_{\ell}(J_{\Sigma_\infty})$ and $M:=T_{\ell}(J_{S_\infty})(-n)$.
 These are free $\mathbb Z_{\ell}$--modules of finite rank. We have to show that they also satisfy the finiteness of $M_{{\Gamma_K}}$ hypothesis. However, for any
$M$ as in the Lemma, $M_{{\Gamma_K}}$ is finite if and only if the action of $\gamma_K$ on the $\mathbb Q_{\ell}$--vector space $\mathbb Q_{\ell}\otimes_{\mathbb Z_{\ell}}M$ has no fixed points
(i.e. it does not have $1$ as an eigenvalue.) Since $J_{\Sigma_\infty}$ is an extension of the Jacobian $J_{X_\infty}$ by a torus, the Riemann Hypothesis for the smooth, projective curve $X_\infty$ over
$\mathbb F$
implies that the eigenvalues of the action of $\gamma_K$ on $\mathbb Q_{\ell}\otimes_{\mathbb Z_{\ell}}T_{\ell}(J_{S_\infty})$ are algebraic integers (independent of $\ell$, if $\ell\ne p$)
of absolute value $r$ (the torus contribution) or $r^{1/2}$ (the Jacobian contribution.) The same argument applied to $\mathbb Q_{\ell}\otimes_{\mathbb Z_{\ell}}T_{\ell}(J_{S_\infty})(-n)$
leads to eigenvalues which are algebraic numbers of absolute value either $r^{1-n}$ (torus) or $r^{1/2-n}$ (Jacobian). Since $n\geq 2$, none of these eigenvalues equals $1$,
so $M_{{\Gamma_K}}$ is finite in both cases.

(1) The Lemma above applied to $M:=T_{\ell}(J_{\Sigma_\infty})$, combined with Corollary \ref{galois-semiabelian} give the following isomorphisms of $\mathbb Z_{\ell}[G]$--modules.
$$T_{\ell}(J_{\Sigma_\infty})^\ast_{\,\,\,{\Gamma_K}}\simeq T_{\ell}(J_{\Sigma_\infty})_{{\Gamma_K}}^{\,\,\vee}\simeq{{\rm Pic}^0_\Sigma(K)^{(\ell)}}^{\, \vee}.$$
This concludes the proof of (1).

(2) The Lemma above applied to $M:=T_{\ell}(J_{\overline S})(-n)$, combined with Lemma \ref{etale-semiabelian} give the following isomorphisms of $\mathbb Z_{\ell}[G]$--modules.
$$T_{\ell}(J_{S_\infty})(-n)^\ast_{\,\,\, {\Gamma_K}}\simeq T_{\ell}(J_{S_\infty})(-n)_{{\Gamma_K}}^{\,\,\vee}\simeq H^2_{et}(O_{K, S}, \mathbb Z_{\ell}(n))\,.$$
This concludes the proof of part (2).
\end{proof}

\begin{theorem}\label{refined-Brumer-Stark} (``refined Brumer-Stark conjecture'') Let ${\rm Pic}^0_\Sigma(K)^{\,\vee}:= {\rm Hom}_{\mathbb Z}({\rm Pic}^0_\Sigma(K), \mathbb Q/\mathbb Z)$, endowed
with the co-variant $G$--action. Then, we have
$$\Theta_{S, \Sigma}(1)\in{\rm Fit}_{\mathbb Z[G]}({\rm Pic}^0_\Sigma(K)^{\, \vee})\,. $$
\end{theorem}
\begin{proof} Since $\Theta_{S, \Sigma}(1)\in\mathbb Z[G]$ (see \S 4.2), it suffices to prove that
$\Theta_{S, \Sigma}(1)\in{\rm Fit}_{\mathbb Z_{\ell}[G]}({{\rm Pic}^0_\Sigma(K)^{(\ell)}}^{\, \vee})$,
for all prime numbers $\ell$, where ${{\rm Pic}^0_\Sigma(K)^{(\ell)}}^{\, \vee}:={\rm Hom}_{\mathbb Z_\ell}({\rm Pic}^0_\Sigma(K)^{(\ell)}, \mathbb Q_\ell/\mathbb Z_\ell)$
with the co-variant $G$--action. Now, we bring the $1$--motive $M_{ S_\infty, \Sigma_\infty}$ into the game (see the paragraph leading to Corollary \ref{Fitting-connected}.)
Let us fix a prime number $\ell$.
We combine Corollary \ref{Fitting-connected} with
Corollary \ref{fitting-duals}(2) to get
\begin{equation}\label{?}
\Theta_{S, \Sigma}(\gamma^{-1})\in{\rm Fit}_{\mathbb Z_{\ell}[[G_\infty]]}(T_{\ell}(M_{S_\infty, \Sigma_\infty})^\ast)\,,\end{equation}
where the $\mathbb Z_{\ell}$--dual is endowed with the co-variant $G_\infty$--action. (Recall that $G_\infty\simeq \Gamma\times H$.) We have the obvious exact sequence
of co-variant $\mathbb Z_{\ell}$--duals in the category of $\mathbb Z_{\ell}[[G_\infty]]$--modules
$$\xymatrix{
0\ar[r] &({\rm Div}^0(S_\infty)\otimes\mathbb Z_{\ell})^\ast \ar[r] & T_{\ell}(M_{S_\infty, \Sigma_\infty})^\ast\ar[r] & T_{\ell}(J_{\Sigma_\infty})^\ast\ar[r] &0
}\,.$$
The behaviour of Fitting ideals with respect to surjections combined with (\ref{?}) above leads to
$$\Theta_{S, \Sigma}(\gamma^{-1})\in {\rm Fit}_{\mathbb Z_{\ell}[[G_\infty]]}(T_{\ell}(J_{\Sigma_\infty})^\ast).$$
Now, we extend scalars along the surjective ring morphism $\pi: \mathbb Z_{\ell}[[G_\infty]]\twoheadrightarrow\mathbb Z_{\ell}[G]$ and obtain
$$\pi(\Theta_{S, \Sigma}(\gamma^{-1}))\in {\rm Fit}_{\mathbb Z_{\ell}[G]}(T_{\ell}(J_{\Sigma_\infty})^\ast_{\,\,\,{\Gamma_K}})\,.$$
However, the ring morphism $\pi$ is induced by the composition of group morphisms $G_\infty\subseteq \overline G=G\times\Gamma\twoheadrightarrow G$,
where $G\time\Gamma\twoheadrightarrow G$ is the projection onto the second factor sending $\gamma$ to $1$. Consequently, we have
$\pi(\Theta_{S, \Sigma}(\gamma^{-1}))=\Theta_{S,\Sigma}(1).$ Now, we apply Corollary \ref{iso-duals}(1) to conclude the proof of the Theorem.
\end{proof}

\begin{remark} Observe that the theorem above is indeed a refinement of the Brumer-Stark Conjecture. Indeed,
since we are dealing with co-variant $\mathbb Z_{\ell}$--duals, we have
$${\rm Fit}_{\mathbb Z[G]}({\rm Pic}^0_\Sigma(K)^{\, \vee})\subseteq {\rm Ann}_{\mathbb Z[G]}({\rm Pic}^0_\Sigma(K)^{\, \vee})={\rm Ann}_{\mathbb Z[G]}({\rm Pic}^0_\Sigma(K))\,,$$
and the inclusion above is strict, in general.
\end{remark}
\medskip

\begin{theorem}\label{refined-Coates-Sinnott} (``refined Coates--Sinnott conjecture'') For all primes $\ell\ne p$ and all $n\in\mathbb Z_{\geq 2}$,
$$\Theta_{S, \Sigma}(q^{n-1})\in {\rm Fit}_{\mathbb Z_{\ell}[G]}(H_{et}^2(O_{K, S}, \mathbb Z_{\ell}(n))).$$
\end{theorem}
\begin{proof} Let us fix a prime $\ell\ne p$ and an integer $n\geq 2$. There is a perfect $\mathbb Z_{\ell}[[G_\infty]]$-linear pairing
$$T_{\ell}(M_{S_\infty, \Sigma_\infty})\times T_{\ell}(M_{\Sigma_\infty, S_\infty})\longrightarrow \mathbb Z_{\ell}(1),$$
generalizing the Weil pairing on $J_{X_\infty}$ (the case $S=\Sigma=\emptyset$) -- see \S10.2 in \cite{Deligne-HodgeIII} for the general theory, but also
see our upcoming work \cite{Greither-Popescu2} for an explicit construction (formula) of such a pairing. The pairing above induces an isomorphism of $\mathbb Z_{\ell}[[G_\infty]]$--modules
$$T_{\ell}(M_{S_\infty, \Sigma_\infty})\overset\sim\longrightarrow T_{\ell}(M_{\Sigma_\infty, S_\infty})^\ast(1)\,,$$
where the $\mathbb Z_{\ell}$--dual is endowed with the contra-variant $G_\infty$--action. Now, apply Corollary \ref{Fitting-connected}(2) combined with
Remark \ref{fitting-twisting} and the above isomorphism to conclude that
$$t_{1-n}(\Theta_{S, \Sigma}(\gamma^{-1}))\in {\rm Fit}_{\mathbb Z_{\ell}[[G_\infty]]}(T_{\ell}(M_{S_\infty, \Sigma_\infty})(n-1))={\rm Fit}_{\mathbb Z_{\ell}[[G_\infty]]}(T_{\ell}(M_{\Sigma_\infty, S_\infty})(-n)^\ast)\,.$$
Above, we used the obvious equality $T_{\ell}(M_{\Sigma_\infty, S_\infty})(-n)^\ast=T_{\ell}(M_{\Sigma_\infty, S_\infty})^\ast(n)$ of $G_\infty$--contra-variant $\mathbb Z_{\ell}$--duals.
Now, we have an obvious exact sequence of $\mathbb Z_{\ell}[[G_\infty]]$--modules.
$$\xymatrix{
0\ar[r] & ({\rm Div}^0(\Sigma_\infty)\otimes\mathbb Z_{\ell})(-n)^\ast\ar[r] & T_{\ell}(M_{\Sigma_\infty, S_\infty})(-n)^\ast\ar[r] & T_{\ell}(J_{S_\infty})(-n)^\ast \ar[r] & 0
}$$
Consequently, as in the proof of the previous theorem, we get the following:
$$t_{1-n}(\Theta_{S, \Sigma}(\gamma^{-1}))\in{\rm Fit}_{\mathbb Z_{\ell}[[G_\infty]]}(T_{\ell}(J_{S_\infty})(-n)^\ast),\ \
 \pi(t_{1-n}(\Theta_{S, \Sigma}(\gamma^{-1}))\in{\rm Fit}_{\mathbb Z_{\ell}[G]}(T_{\ell}(J_{S_\infty})(-n)^\ast_{\,\,\, {\Gamma_K}})\,.$$
As before, the second relation above is obtained from the first via extension of scalars along the natural surjective ring morphism $\pi: \mathbb Z_{\ell}[[G_\infty]]\twoheadrightarrow\mathbb Z_{\ell}[G]$ induced by the composition of group morphisms
$G_\infty\subseteq G\times\Gamma\twoheadrightarrow G$. It is easy to see that
$$\pi(t_{1-n}(\Theta_{S, \Sigma}(\gamma^{-1})))=\pi(\Theta_{S, \Sigma}(q^{n-1}\cdot\gamma^{-1}))=\Theta_{S, \Sigma}(q^{n-1})$$
(see comments leading to Definition \ref{Tate-twists}.)
Now, one combines this equality with the second relation above and the isomorphism of
Corollary \ref{iso-duals}(2) to conclude the proof of the Theorem.
\end{proof}

\section{\bf Cartier operators, logarithmic differentials and $p$--torsion of Picard $1$--motives}

\noindent The main goal of this final section is to show how our
Theorem \ref{ml-free} implies the main result of Nakajima's paper
\cite{Nakajima} on Galois $p$--covers of smooth, connected,
projective curves defined over an algebraically closed field of
characteristic $p$. In the process, we place Nakajima's result in
the general framework of Picard $1$--motives.

\subsection{\bf Nakajima's Theorem} In what follows, $Z$ will denote a connected, smooth, projective curve over
an algebraically closed field $\kappa$ of characteristic $p$. Let $\mathcal K:=\kappa(Z)$ be the field of $\kappa$--rational functions of $Z$. We denote by $\Omega_Z$ the
$\mathcal K$--module $\Omega_{\mathcal K/\kappa}$ of K\"ahler differentials (which can also be identified with the generic fiber of the sheaf of differentials on $Z$.)
$G$ will denote a finite group of  $\kappa$--automorphisms of $Z$. Obviously, $\Omega_Z$ inherits a $\kappa[G]$--module structure, with the canonical
$G$--action ${}^\sigma(x\cdot dy)={}^\sigma x\cdot d({}^\sigma y)$, for all $\sigma\in G$ and all $x, y\in\mathcal K$.
We let $\mathcal S$ denote a $G$--invariant finite set of closed points in $Z$ and let $\mathcal S':=\pi(\mathcal S)$ (the set of closed points on $Z'$ lying below those in $\mathcal S$.)
We take the effective divisor $[\mathcal S]=\sum_{P\in\mathcal S}P$ on $Z$ and let
$$\Omega_{Z}(-[\mathcal S]):=\left\{\omega\in\Omega_{Z}\mid \text{div}(\omega)\geq -[\mathcal S]\right \}$$
be the space of differentials on $Z$ whose divisors are greater than or equal to $-[\mathcal S]$.

\begin{definition} The Cartier operator
$\mathcal C: \Omega_{Z}\longrightarrow \Omega_{Z}$
is the unique $\mathbb F_p$--linear map, satisfying the following properties.
\begin{enumerate}
\item $\mathcal C(x^p\omega)=x\mathcal C(\omega),\, \forall\, x\in\mathcal K$ and $\forall\, \omega\in \Omega_{Z}$, i.e. $\mathcal C$ is $p^{-1}$ -- linear.
\item $\mathcal C(df)=0,\, \forall\, f\in\mathcal K\,.$
\item $\mathcal C(x^{p-1}dx)=dx,\, \forall\, x\in\mathcal K\,.$
\end{enumerate}
\end{definition}
\noindent The reader may consult \cite{Nakajima} and \cite{Subrao} for the existence and uniqueness of $\mathcal C$. The uniqueness of $\mathcal C$ implies right away that $\mathcal C$ is
an $\mathbb F_p[G]$--linear map. Also, since $\mathcal S$ is $G$--invariant, Lemma 2.1 in \cite{Subrao} implies that $\mathcal C$ induces an $\mathbb F_p[G]$--linear map
$$\mathcal C:\Omega_{Z}(-[\mathcal S])\longrightarrow \Omega_{Z}(-[\mathcal S])\,.$$

\begin{definition} Let $\Omega_{Z}(-[\mathcal S])^{\mathcal C=1}$ be the $\mathbb F_p$--vector subspace of $\Omega_{Z}(-[\mathcal S])$
consisting of all the differentials $\omega$ which are fixed by the Cartier operator $\mathcal C$.
Let $\Omega_{Z}(-[\mathcal S])^s$ be the $\kappa$--vector subspace of $\Omega_{Z}(-[\mathcal S])$
spanned by $\Omega_{Z}(-[\mathcal S])^{\mathcal C=1}\,.$ $\Omega_{Z}(-[\mathcal S])^s$ is called the semi-simple part of
$\Omega_{Z}(-[\mathcal S])$ with respect to the $p^{-1}$--linear operator $\mathcal C$.
\end{definition}

\begin{remark} Obviously, if $G$ is a group of automorphisms of $Z$ fixing $\mathcal S$, then $\Omega_{Z}(-[\mathcal S])^{\mathcal C=1}$ is an $\mathbb F_p[G]$--module and
$\Omega_{Z}(-[\mathcal S])^s$ is a $\kappa[G]$--module. Also,
it turns out that one has a direct sum decomposition in the category of $\kappa[G]$--modules
$$\Omega_{Z}(-[\mathcal S])=\Omega_{Z}(-[\mathcal S])^s\oplus \Omega_{Z}(-[\mathcal S])^n,$$
where $\Omega_{Z}(-[\mathcal S])^n$ is the nilpotent space associated to $\mathcal C$, consisting of all
differential forms $\omega\in \Omega_{Z}(-[\mathcal S])$ which are killed by a power of $\mathcal C$
(see Theorem 2.2 in \cite{Subrao}.)\end{remark}

\begin{theorem}\label{Nakajima} {\bf (Nakajima)} Assume that $\pi: Z\mapsto Z'$ is a $G$--Galois cover of connected, smooth, projective curves
over $\kappa$, where $G$ is a finite $p$--group. Assume that $\mathcal S$ is a finite, non--empty,
$G$--equivariant set of closed points on $Z$, containing
the ramification locus for $\pi$. Then $\Omega_Z(-[\mathcal S])^s$ is a free $\kappa[G]$--module whose rank is given by
$${\rm rank}_{\kappa[G]}\Omega_Z(-[\mathcal S])^s =(\gamma_{Z'}-1+|\mathcal S'|)\,.$$
\end{theorem}
\begin{proof} See \cite{Nakajima}, Theorem 1, p. 561.\end{proof}

\subsection{Nakajima's theorem in the language of $p$--torsion of Picard $1$-motives} For the moment, let us assume that $Z\to Z'$ is a $G$--Galois cover of smooth, connected,  projective curves defined over $\kappa$
and that $\mathcal S$ is a $G$--invariant, finite, non-empty  set of closed points on $Z$, containing the ramification locus of the cover.
Next, we make a connection between
the $\kappa[G]$--module  $\Omega_Z(-[\mathcal S])^s$ and the $\mathbb F_p[G]$--module $\mathcal M_{\mathcal S, \emptyset}[p]$
of $p$--torsion points of the Picard $1$--motive $\mathcal M_{\mathcal S, \emptyset}$ associated to the data $(Z, \kappa, \mathcal S, \emptyset)$.
As a consequence, we show how Nakajima's Theorem \ref{Nakajima} is in fact equivalent to a particular case of our Theorem \ref{ml-free}.
\begin{lemma}\label{Cartier}  The following hold true.
\begin{enumerate}
\item Let $\omega\in \Omega_{Z}(-[\mathcal S])$. Then $\omega\in \Omega_{Z}(-[\mathcal S])^{\mathcal C=1}$ if and only if there exists $f\in\mathcal K^\times$, such that $\omega=\frac{df}{f}$.
\item Let $f\in\mathcal K^\times$. Then $\frac{df}{f}\in\Omega_{Z}(-[\mathcal S])$ if and only if $f\in \mathcal K_{\mathcal S, \emptyset}^{(p)}$.
\item There exists an isomorphism of $\mathbb F_p[G]$--modules
$$\mathcal K_{\mathcal S, \emptyset}^{(p)}/\mathcal K_{\emptyset}^{\times p} \simeq \Omega_{Z}(-[\mathcal S])^{\mathcal C=1}$$
given by $\widehat f \mapsto \frac{df}{f}\,,$ for all $f\in \mathcal K_{\mathcal S, \emptyset}^{(p)}$.
\item There exists an isomorphism of $\mathbb F_p[G]$--modules
$$\mathcal M_{\mathcal S, \mathcal T}[p]\simeq \Omega_{Z}(-[\mathcal S])^{\mathcal C=1}$$
for all $G$--invariant, finite (possibly empty) sets $\mathcal T$ of closed points on $Z$, disjoint from $\mathcal S$.
\end{enumerate}
\end{lemma}
\begin{proof} For (1) see ``FACT'' in \cite{Subrao}, p. 4, for example. The proof of (2) can also be found in \cite{Subrao},
but since it is straightforward,
for the convenience of the reader we sketch it next. Let $f\in\mathcal K^\times$ and let $P$ be a closed point on $Z$. Let $t\in \mathcal K$ be a uniformizer
at $P$ and let $\mathcal K_P\simeq\kappa((t))$ be the completion of $\mathcal K$ at $P$. Assume that $\text{ord}_P(f)=n$ and
$f=a_nt^n + a_{n+1}t^{n+1} +\cdots$ in $\kappa((t))$, with $a_n, a_{n+1},\dots\in\kappa$ and $a_n\ne 0$. Then the image of $\omega:=\frac{df}{f}$
via the natural embedding $\Omega_Z\hookrightarrow\Omega_{\kappa((t))/\kappa}\simeq\kappa((t))dt$ is given by
$$\frac{df}{f}=(\frac{n}{t} + b_0 + b_1t +\cdots)dt,$$
with $b_0, b_1, \dots$ in $\kappa$. This shows that $\frac{df}{f}\in\Omega_{Z}(-[\mathcal S])$ if and only if $p\mid\text{ord}_P(f)$, for all $P\not\in\mathcal S$,
i.e. if and only if $f\in \mathcal K_{\mathcal S, \emptyset}^{(p)}$.

Now, by first noting that $\frac{d(fg)}{fg}=\frac{df}{f}+\frac{dg}{g}$, for all $f, g\in \mathcal K^\times$, one easily sees that (1) and (2) imply that we have a surjective $\mathbb Z[G]$--module
morphism  $\mathcal K_{\mathcal S, \emptyset}^{(p)}\twoheadrightarrow \Omega_{Z}(-[\mathcal S])^{\mathcal C=1}$, given by $f\rightarrow\frac{df}{f}\,.$ The kernel of this morphism
consists of all $f\in \mathcal K_{\mathcal S, \emptyset}^{(p)}$, such that $df=0$. This kernel is easily seen to equal $\mathcal K^{\times p}_{\emptyset}$, by picking a uniformizer
$t$ at an arbitrary closed point $P$ on $Z$ and (canonically) embedding $\Omega_Z\simeq\mathcal K dt$ into $\Omega_{\kappa((t))/\kappa}\simeq\kappa((t))dt$, as before.
This concludes the proof of (3).

In order to prove part (4), fix $\mathcal T$ as above. If one combines Remark \ref{p-torsion} and Proposition \ref{interpret} with part (3) of our Lemma, one obtains the desired isomorphism.
\end{proof}

\begin{lemma}\label{tensor} The canonical $\kappa$--linear map
$$\Omega_{Z}(-[\mathcal S])^{\mathcal C=1}\otimes_{\mathbb F_p}\kappa \longrightarrow \Omega_{Z}(-[\mathcal S])^s,\qquad \omega\otimes x\to x\omega$$
induces an isomorphism of $\kappa[G]$--modules.
\end{lemma}
\begin{proof} It suffices to show that if $\{\omega_1, \dots , \omega_n\}$ are $\mathbb F_p$--linearly independent
elements of $\Omega_{Z}(-[\mathcal S])^{\mathcal C=1}$, then they remain $\kappa$--linearly independent in
$\Omega_{Z}(-[\mathcal S])$. Assume that this is not true. Without loss of generality, one may assume that $\{\omega_1,\dots,\omega_{n-1}\}$ are
$\kappa$--linearly independent, but
$$\omega_n=x_1\cdot \omega_1 + \cdots + x_{n-1}\cdot\omega_{n-1}\,,$$
for some $x_1,\dots,x_{n-1}\in\kappa$. Now, apply $\mathcal C$ to
the equality above. We obtain
$$\omega_n=x_1^{1/p}\cdot \omega_1 + \cdots + x_{n-1}^{1/p}\cdot\omega_{n-1}\,.$$
This implies (via the assumed $\kappa$--linear independence of
$\omega_1, \dots, \omega_{n-1}$) that
$$x_1=x_1^{1/p}, \dots,
x_{n-1}=x_{n-1}^{1/p}\,.$$
Consequently,  $x_1, \dots, x_{n-1}\in
\mathbb F_p$, which contradicts the $\mathbb F_p$--linear independence of
$\{\omega_1, \dots, \omega_n\}$.
\end{proof}

Next, we state and prove the main result of this section.

\begin{proposition}\label{equivalence} Let $\kappa$ be an algebraically closed field of characteristic $p$. Let $Z\to Z'$
be a $G$--Galois cover of connected, smooth, projective curves defined over $\kappa$. Let $\mathcal S$ and $\mathcal T$ be
$G$--equivariant, disjoint, finite sets of closed points on $Z$, such that $S$ is non-empty and contains the ramified locus of the cover.
Assume that $G$ is a $p$--group. Then, the following are equivalent.
\begin{enumerate}
  \item {\bf (Nakajima's Theorem)} $\Omega_Z(-[\mathcal S])^s$ is a free $\kappa[G]$--module whose rank is given by
$${\rm rank}_{\kappa[G]}\Omega_Z(-[\mathcal S])^s =(\gamma_{Z'}-1+|\mathcal S'|)\,.$$
  \item {\bf (Theorem \ref{ml-free} for $\ell=p$)} $\mathcal M_{\mathcal S, \mathcal T}[p]$ is a free $\mathbb F_p[G]$--module whose rank is given by
$${\rm rank}_{\mathbb F_p[G]}\mathcal M_{\mathcal S, \mathcal T}[p] =(\gamma_{Z'}-1+|\mathcal S'|)\,.$$
\end{enumerate}
\end{proposition}
\begin{proof} Lemma \ref{tensor} gives an isomorphism of $\kappa[G]$--modules
$$\Omega_{Z}(-[\mathcal S])^{\mathcal C=1}\otimes_{\mathbb F_p[G]}\kappa[G] \overset\sim\longrightarrow \Omega_{Z}(-[\mathcal S])^s,\qquad \omega\otimes x\to x\omega\,.$$
Combine this isomorphism with the fact that $\kappa[G]$ is a faithfully flat $\mathbb F_p[G]$--algebra to conclude that
(1) above is equivalent to the statement
\medskip

\rm{(1')}\quad $\Omega_Z(-[\mathcal S])^{\mathcal C=1}$ is a free $\mathbb F_p[G]$--module of rank
$(\gamma_{Z'}-1+|\mathcal S'|)\,.$
\medskip

\noindent Now, use the isomorphism in Lemma \ref{Cartier} (4) to conclude that statement (1') above is equivalent to statement (2)
in our Proposition.
\end{proof}

\bibliographystyle{plain}
\bibliography{greither-popescu}
\nocite{Cassels-Frohlich}
\nocite{Crew-Katz}
\nocite{Crew-pcovers}
\nocite{Illusie-deRham}

\end{document}